\documentclass[onefignum,onetabnum]{siamart250211}
\usepackage{bm}
\usepackage{mathrsfs}
\usepackage{amssymb}
\usepackage{graphicx}
\usepackage{subfigure}
\usepackage{color}

\usepackage{amsmath}
\allowdisplaybreaks

\usepackage{subfigure}
\usepackage{graphicx}
\usepackage{color}
\usepackage{mathrsfs}

\usepackage{appendix} 

\def\dfrac{\displaystyle\frac}

\usepackage{lipsum}
\usepackage{amsfonts}
\usepackage{graphicx}
\usepackage{epstopdf}
\usepackage{algorithmic}
\ifpdf
  \DeclareGraphicsExtensions{.eps,.pdf,.png,.jpg}
\else
  \DeclareGraphicsExtensions{.eps}
\fi

\usepackage{enumitem}
\setlist[enumerate]{leftmargin=.5in}
\setlist[itemize]{leftmargin=.5in}

\newsiamremark{example}{Example}
\newsiamremark{remark}{Remark}
\newsiamremark{Condition}{Condition}
\newsiamremark{assumption}{Assumption}
\newsiamthm{claim}{Claim}

\title{High-Probability Sampled-Data Stabilization of General Nonlinear Stochastic Systems}

\author{Kele Sun\thanks{Research Institute of Intelligent Complex Systems, Fudan University, Shanghai, 200433, China.}   
\and Yudong Wang\thanks{Research Institute of Intelligent Complex Systems, Fudan University, Shanghai, 200433, China.}
  \and Shijie Zhou\thanks{Corresponding author, Research Institute of Intelligent Complex Systems, Fudan University, Shanghai, 200433, China(\email{sjzhou14@fudan.edu.cn}).}  
  \and Wei Lin\thanks{School of Mathematics Sciences, Fudan University, Shanghai, 200433, China.}  
  \and Xuerong Mao  \thanks{Department of Mathematics and Statistics,
University of Strathclyde, G11XH Glasgow, U.K.}}

\usepackage{amsopn}

\ifpdf
\hypersetup{
  pdftitle={An Example Article},
  pdfauthor={D. Doe, P. T. Frank, and J. E. Smith}
}
\fi

\begin{document}
\begin{sloppypar}
 \maketitle
\begin{abstract}
In this paper, we study sampled-data stabilization of general nonlinear stochastic systems under a local exponential-type Lyapunov condition for the continuous-time closed-loop system. The coefficients are assumed {\it only locally Lipschitz, with neither linear-growth nor Khasminskii-type conditions}. Two counterexamples show that almost-sure stabilization is generally unattainable. Instead, given an initial state, we establish exponential stability with arbitrarily high probability under sufficiently fast sampling. Separate methods are developed for \(p\ge 2\) and \(0<p<2\), where \(p\) is the growth order in the Lyapunov condition. Numerical examples illustrate the results.
\end{abstract}

\section{Introduction}\label{sec_1}
Consider a controlled stochastic differential equation(SDE)
\begin{small}\begin{equation}\label{eq_2_1}
\mathrm{d}\bm{x}_t
=
\bm{f}(\bm{x}_t,\bm{\alpha}(\bm{x}_t))\,\mathrm{d}t
+
\bm{g}(\bm{x}_t,\bm{\alpha}(\bm{x}_t))\,\mathrm{d}B_t,
\end{equation}\end{small}
with feedback $\bm{\alpha}(\cdot)$ to be designed. We seek $\bm{\alpha}(\cdot)$ that stabilizes the closed-loop system. In practice, control signals are held constant over sampling intervals, yielding the sampled-data system
\begin{small}\begin{equation}\label{eq_2_2}
\mathrm{d}\bm{y}_t
=
\bm{f}(\bm{y}_t,\bm{\alpha}(\bm{y}_{\delta_t}))\,\mathrm{d}t
+
\bm{g}(\bm{y}_t,\bm{\alpha}(\bm{y}_{\delta_t}))\,\mathrm{d}B_t,
\quad
\delta_t=\left\lfloor \frac{t}{\tau}\right\rfloor\tau,
\end{equation}\end{small}
where $\tau>0$ is the sampling period. Does \eqref{eq_2_2} preserve the stability of \eqref{eq_2_1}?

This is the problem of sampled-data control, which has attracted much attention for reducing computational burden. Mao \cite{Mao2013,mao2015almost} initiated this direction, followed by numerous studies \cite{HuLiuDengMao2020,LiMaoMukamaYuan2020,MaoLamHuang2008,MaoLiuHuLuoLu2014,QiuLiuHuMaoYou2016,YouLiuLuMaoQiu2015}. Event-triggered stabilization for stochastic systems was studied in \cite{8897000,9216607,liu2025stabilization,10541043,zhu2024event}. Most of these works assume linear growth conditions on the drift and diffusion coefficients. Several works have addressed highly nonlinear SDEs under Khasminskii-type conditions \cite{DongTangMao2022,8789495,LiMao2020HighlyNonlinear,MeiFeiFeiMao2020,mao22_3,XuMaoYin2026}. Sampled-data homogeneous feedback was used in \cite{10480561,YuLin2024,YuLin2025} for stochastic systems with homogeneity degree zero. Recently, \cite{LiuTeelSun2022,LiuTeelSun2024} established a general framework covering both sampled-data and event-triggered control, with stochastic transmission times and protocols.

Although these works cover a broad class of stochastic systems, they require restrictive conditions that are often hard to verify or too conservative. A natural question arises,  {\it can we obtain stability under a weaker local Lipschitz condition?} Some progress has been made for ordinary differential equations(ODEs). For ODEs ($\bm{g}\equiv 0$ in \eqref{eq_2_1}, \eqref{eq_2_2}), Lin \cite[Theorem~3.1]{LinWei_TAC2020} proved that if \eqref{eq_2_1} is exponentially stable under $\bm{\alpha}$, then, for each fixed $r>0$, there exists $\tau^*(r)>0$ such that \eqref{eq_2_2} is exponentially stable for all $\|\bm y_0\|\le r$ whenever $0<\tau<\tau^*(r)$. The proof only requires $\bm{f}$ and $\bm{\alpha}$ to be {\it locally Lipschitz, without global growth conditions}.

In the proof of \cite[Theorem~3.1]{LinWei_TAC2020}, the authors used a ``level set" argument: for any initial value in a compact set, the trajectory stays in a larger bounded set when $\tau$ is  enough, where the Lipschitz constant is bounded. This yields a Halanay-type inequality \cite[Eq. (82)]{LinWei_TAC2020} and hence the exponential stability.

This raises several questions.

\medskip
\noindent\textbf{Question 1.} 
Can the result of \cite{LinWei_TAC2020} be extended to SDEs in the almost sure sense?

\noindent\textbf{Answer.} 
No. Even if the initial value lies in a compact set and $\tau$ is arbitrarily small, there is always a nonzero probability that a large Brownian increment $\Delta B_t$ drives the state outside the bounded region. To illustrate this more clearly, we provide two counterexamples below; the first is inspired by \cite[Example 2.2]{10480561}.

\begin{example}[First counterexample: non-explosive continuous case but explosive under sampling]
\label{ex_1_1}
Consider the one-dimensional stochastic control system
\begin{small}\begin{equation}\label{ex_1_1_1}
    \mathrm{d}x_t
    =
    \left(x_t^2+x_t^3+u_t\right)\,\mathrm{d}t
    +
    x_t^2\,\mathrm{d}B_t,
    \quad x_0>0.
\end{equation}\end{small}
The model combines a superlinear drift \(x^2+x^3\), which can cause finite-time blow-up (as in fluid collapse and population aggregation), with multiplicative noise \(x^2\,\mathrm{d}B_t\), whose quadratic intensity produces large fluctuations at high states; \(u_t\) is an external stabilizing control \cite{ide2002oscillatory}. With control \textcolor{black}{$u_t=\alpha(x_t)\triangleq-a x_t - (k+1) x_t^3 - x_t^2 $ (\(a>0, k>1/2\))}, the closed-loop system becomes
$    \mathrm{d}x_t = (-a x_t - k x_t^3)\,\mathrm{d}t + x_t^2\,\mathrm{d}B_t$,
and hence
$
    \mathbb{E}|x_t|^2 \le |x_0|^2 {\rm e}^{-2a t}.
$
Thus \(\alpha(x)\) stabilizes the system in  mean square.

However, under sampled-data control, the same feedback causes explosion in arbitrarily short time with positive probability (proved in {\bf Appendix B}, Lemma \ref{ex_5_1}), ruling out almost sure stabilization.
\end{example}


\begin{example}[Second counterexample: non-explosive, but with some oscillatory divergent paths]
\label{ex_2_2}
Consider the one-dimensional stochastic control system
\begin{small}\[
    \mathrm{d}x_t
    =
    \left( x_t+u_t\right)\,\mathrm{d}t
    +x_t^2\,\mathrm{d}B_t,
    \quad x_0>0.
\]\end{small}
\textcolor{black}{This is a controlled CEV-type model whose diffusion term \(x^2\,\mathrm{d}B_t\) makes the volatility grow quadratically with the state, capturing the sharp rise in uncertainty at high state levels. Such models are widely used in asset pricing, risk management, and stochastic optimal control \cite{cox1996constant}}. With feedback
\textcolor{black}{$ u_t=\alpha(x_t)\triangleq-\kappa x_t-x_t^3$
 ($\kappa>1$)}, the closed-loop system becomes
$
    \mathrm{d}x_t
    =[-(\kappa-1) x_t-x_t^3]\,\mathrm{d}t
    +x_t^2\,\mathrm{d}B_t,
$
and hence
$
    \mathbb{E}|x_t|^2
    \le
    |x_0|^2 {\rm e}^{-2(\kappa-1)t}.
$
Thus \(\alpha(x)\) stabilizes the system in mean square.

For every $\tau>0$, the solution exists globally but, with positive probability, \textcolor{black}{diverges at the sampling times while the trajectory crosses zero infinitely often} (see Appendix~B, Lemma~\ref{ex_5_2}), ruling out almost-sure stability.
\end{example}

The above examples show that almost sure stability is generally unattainable for sampled-data systems. Relaxing this to stability with probability \(1-\varepsilon\), we adopt the level-set idea of \cite{LinWei_TAC2020} combined with stopping times; the key is to bound the exit probability from a level set by \(\varepsilon\).

A second question concerns the proof of \cite[Theorem~3.1]{LinWei_TAC2020}.

\medskip
\noindent\textbf{Question 2.} 
The proof of \cite[Theorem~3.1]{LinWei_TAC2020} employs a Halanay-type inequality \cite[Eq. (82)]{LinWei_TAC2020} to derive exponential stability. Can this method be generalized to SDEs?

\noindent\textbf{Answer.} 
Not necessarily. The ODE proof requires a \(C^1\) Lyapunov function, while Itô's formula for SDEs requires \(C^2\) due to the diffusion term. This smoothness may fail even for almost surely exponentially stable systems. Consider the scalar SDE
$
\mathrm{d}x_t = 0.2 x_t\,\mathrm{d}t + x_t\,\mathrm{d}B_t,
$
with solution \(x_t = x_0 \exp(-0.3 t + B_t)\). Then
$
\lim_{t\to\infty} \frac{1}{t}\log |x_t| = -0.3 \quad \text{a.s.},
$
so the system is almost surely exponentially stable. For \(V_p(x)=|x|^p\), we have
$
\mathcal{L}V_p(x) = \frac{p(p-0.6)}{2}|x|^p,
$
so \(\mathbb{E}|x_t|^p \to 0\) for \(p<0.6\), but \(\mathbb{E}|x_t|^p \to \infty\) for \(p>0.6\). Yet for any \(p<0.6\), \(V_p\) is not \(C^2\) at \(x=0\). Thus the Halanay method is inapplicable despite almost sure stability.

Our approach is divided into two cases depending on the regularity of the Lyapunov function.

When a \(C^2\) Lyapunov function is available, we employ a Halanay-type inequality. However, the stopped process may render the Lyapunov function \(U(t)\) (see \eqref{eq_3_8}) discontinuous. We therefore extend the classical Halanay inequality to a generalized version that permits discontinuous functions (Lemma \ref{lem_3_1}).

When no \(C^2\) Lyapunov function exists, we adopt the comparison method of \cite{Mao2013,mao2015almost}. On each interval, we construct an auxiliary process, estimate the exit probability from a level set via stopping times, and show that these probabilities form a geometric series whose sum is less than any prescribed \(\varepsilon>0\). 

The main contributions are as follows.
\begin{itemize}
\item 
Unlike other works that focus on SDEs with linear-growth or Khasminskii-type conditions \cite{DongTangMao2022,8789495,LiMao2020HighlyNonlinear,Mao2013,mao2015almost,MeiFeiFeiMao2020,mao22_3}, we study sampled-data SDEs with {\it merely locally Lipschitz coefficients} and local exponential-type Lyapunov conditions. We prove exponential stabilization with arbitrarily high probability provided the sampling period is sufficiently small for the given initial state (see Theorems~\ref{thm_3_1} and \ref{thm_4_4}).

\item 
We establish a Halanay-type inequality for possibly discontinuous functions under two one-sided semicontinuity conditions (see Lemma \ref{lem_3_1}), extending the right-continuous result in \cite{XuMaoYin2026} and applying to the Lyapunov quantity arising from the stopping-time argument for \(p\geq2\),where $p$ is the exponent in Assumption \ref{ass_2_2}. 

\item
For the non-\(C^2\) case (\(0<p<2\)), we adopt the comparison method of \cite{Mao2013,mao2015almost} with stopping times.

\item 
Two counterexamples demonstrate the sharpness of our results, showing that almost sure stability cannot generally be expected (see {\bf Appendix B}).
\end{itemize}

The remainder of this paper is organized as follows. Section \ref{sec_2} presents the necessary preliminaries and assumptions on the SDEs. Sections \ref{sec_3} and \ref{sec_4} contain our main results and their proofs. Section~\ref{sec_5} presents two numerical examples to illustrate the high-probability stabilization results.
 Finally, Section~\ref{sec_6} concludes the paper and discusses several directions for future research.

\medskip
\noindent\textbf{Notations.}
Throughout this paper, $\mathbb{N}\triangleq\{1,2,\ldots\}$. For $x\in\mathbb{R}$, $|x|$ is its absolute value. \(\operatorname{sgn}(x)\triangleq 1,0,-1\) for \(x>0\), \(x=0\), and \(x<0\), respectively. For vectors $\bm{x}\in\mathbb{R}^n$ or matrices $\bm{A}\in\mathbb{R}^{n\times r}$, $\|\cdot\|$ denotes the Euclidean or Frobenius norm, and $\bm{A}^{\top}$ is the transpose. For $a,b\in\mathbb{R}\cup\{\pm\infty\}$, $a\wedge b\triangleq\min\{a,b\}$, $a\vee b\triangleq\max\{a,b\}$. For a set $A$, $A^c$ is its complement and $\mathbf{1}_A$ its indicator. For integrable $\xi$ and sub-$\sigma$-algebra $\mathscr G\subset\mathscr F$, $\mathbb E[\xi\mid\mathscr G]$ is the conditional expectation; for $A\in\mathscr F$, $\mathbb P(A\mid\mathscr G)\triangleq\mathbb E[\mathbf{1}_A\mid\mathscr G]$. The abbreviation $\mathrm{a.s.}$ means almost surely. For differentiable $V:\mathbb{R}^n\to\mathbb{R}$, $V_{\bm{x}}$ and $V_{\bm{x}\bm{x}}$ denote its gradient and Hessian w.r.t. $\bm{x}$.

\section{Problem formulation and preliminaries}\label{sec_2}
Consider Eqs. \eqref{eq_2_1}–\eqref{eq_2_2} with state vectors $\bm{x}_t,\bm{y}_t\in\mathbb{R}^n$. Let $B_t$ be a $1$-dimensional Brownian motion on $\left(\Omega,\mathscr{F},\{\mathscr{F}_t\}_{t\ge 0},\mathbb{P}\right)$ with filtration satisfying the usual conditions. The functions
$
\bm{f}:\mathbb{R}^n\times\mathbb{R}^n\to\mathbb{R}^n,\quad
\bm{g}:\mathbb{R}^n\times\mathbb{R}^n\to\mathbb{R}^n,\quad
\bm{\alpha}:\mathbb{R}^n\to\mathbb{R}^n
$
are continuous. We impose the following assumptions.
\begin{assumption}\label{ass_2_1}
For each $R>0$, there is $L_R>0$ such that for all $\bm{x},\bm{x}',\bm{y},\bm{y}'\in\mathbb{R}^n$ with $\max\{\|\bm{x}\|,\|\bm{x}'\|,\|\bm{y}\|,\|\bm{y}'\|\}\le R$,
\begin{small}\begin{align}
\|\bm{f}(\bm{x},\bm{y})-\bm{f}(\bm{x}',\bm{y}')\| \vee \|\bm{g}(\bm{x},\bm{y})-\bm{g}(\bm{x}',\bm{y}')\|
&\le
L_R(\|\bm{x}-\bm{x}'\|+\|\bm{y}-\bm{y}'\|), \label{eq_2_4}\\
\|\bm{\alpha}(\bm{x})-\bm{\alpha}(\bm{x}')\|
&\le
L_R\|\bm{x}-\bm{x}'\|, \label{eq_2_6}\\
\label{eq_2_7}
\bm{f}(\bm{0},\bm{0})=\bm{0},\quad \bm{g}(\bm{0},\bm{0})=\bm{0},\quad \bm{\alpha}(\bm{0})=\bm{0}.
\end{align}\end{small}
\end{assumption}

\begin{remark}\label{rem_2_1}

By Assumption~\ref{ass_2_1}, for every \(R>0\) and $p>0$, there exist
constants
\(M_{f,R}^{(p)},M_{g,R}^{(p)},D_{f,R},D_{g,R},\Lambda_{f,R},\Lambda_{g,R}\), and
\(\Gamma_R\) such that, whenever
\(\|\bm x\|\vee\|\bm z\|\leq R\),
\begin{small}$$\begin{aligned}
&\|\bm f(\bm x,\bm\alpha(\bm z))\|^p
\leq
M_{f,R}^{(p)}
\left(\|\bm x\|^p+\|\bm z\|^p\right),
\|\bm g(\bm x,\bm\alpha(\bm z))\|^p
\leq
M_{g,R}^{(p)}
\left(\|\bm x\|^p+\|\bm z\|^p\right),
\\&\|\bm f(\bm x,\bm\alpha(\bm z))-\bm f(\bm x,\bm\alpha(\bm x))\|
\leq
D_{f,R}\|\bm x-\bm z\|,
\|\bm g(\bm x,\bm\alpha(\bm z))-\bm g(\bm x,\bm\alpha(\bm x))\|
\leq
D_{g,R}\|\bm x-\bm z\|,
\\&\|\bm f(\bm x,\bm\alpha(\bm x))\|
\leq
\Gamma_R\|\bm x\|, \|\bm g(\bm x,\bm\alpha(\bm x))\|
\leq
\Gamma_R\|\bm x\|.\\
&\|\bm  f(\bm x,\bm\alpha(\bm x))- \bm f(\bm z,\bm\alpha(\bm z))\|
\leq
\Lambda_{f,R}\|\bm x-\bm z\|,
\|\bm  g(\bm x,\bm\alpha(\bm x))- \bm g(\bm z,\bm\alpha(\bm z))\|
\leq
\Lambda_{g,R}\|\bm x-\bm z\|
\end{aligned}$$\end{small}

\end{remark}


 We impose a {\it local} exponential-type Lyapunov condition on  system \eqref{eq_2_1}.

\begin{assumption}\label{ass_2_2}
There exist constants \(p>0\), \(c_1>0\), and a function \(V:\mathbb{R}^n\to [0,\infty)\) such that:
\begin{enumerate}
\item[(i)] \(V(\bm 0)=0\), \(V\) is continuous on \(\mathbb{R}^n\), and 
\(V\in C^2(\mathbb{R}^n\setminus\{\bm 0\};[0,\infty))\). 
Moreover, if \(p\ge 2\), then \(V\in C^2(\mathbb{R}^n)\);

\item[(ii)] for all \(\bm x\in\mathbb{R}^n\setminus\{\bm 0\}\),
\begin{small}$
V(\bm x) \ge c_1 \|\bm x\|^p, \mathcal L V(\bm x) < 0;$\end{small}

\item[(iii)] for every \(R>0\), there exists \(\lambda_R>0\) such that for all \(0<\|\bm x\|\le R\),
\begin{small}$
\mathcal L V(\bm x) \le -\lambda_R V(\bm x);$\end{small}

\item[(iv)] for every \(R>0\), there exist constants \(c_{2,R},c_{3,R}>0\) such that for all \(0<\|\bm x\|\le R\),
\begin{small}$
\|V_{\bm x}(\bm x)\| \le c_{2,R}\|\bm x\|^{p-1}, 
\|V_{\bm x\bm x}(\bm x)\| \le c_{3,R}\|\bm x\|^{p-2}.
$\end{small}
\end{enumerate}
Here, for \(\bm x\ne\bm 0\), the infinitesimal generator \(\mathcal L V\) is defined by
\begin{small}\[
\mathcal L V(\bm x) \triangleq
V_{\bm x}(\bm x)^\top \bm f(\bm x,\bm\alpha(\bm x))
+ \frac12\bm g(\bm x,\bm\alpha(\bm x))^\top V_{\bm x\bm x}(\bm x)\,\bm g(\bm x,\bm\alpha(\bm x)).
\]\end{small}
\end{assumption}

\begin{remark}\label{rem_2_2}
By the stochastic LaSalle invariance principle \cite{MAO1999350,mao1999stochastic}, (ii) implies \(\bm{x}_t\to\bm{0}\) a.s. for \eqref{eq_2_1}, while (iii) yields exponential convergence \cite[Chapter~4, Theorem~3.3]{maobook}, essential for the stability of \eqref{eq_2_2}. For ODEs (\(\bm{g}\equiv0\)), exponential stability of \eqref{eq_2_1} yields semiglobal stability of \eqref{eq_2_2}, whereas asymptotic stability yields only practical stability; see \cite{LinWei_TAC2020,shim2003asymptotic}.
\end{remark}

\begin{remark}
The canonical choice \(V(\bm x)=\|\bm x\|^p\) satisfies the regularity and derivative-growth requirements in Assumption~\ref{ass_2_2}(i) and (iv). Condition (iii) requires only {\it local} exponential decay, not a uniform global rate, thus extending \cite[Chapter~4, Theorem~4.4]{maobook}. Indeed, for \(\dot{x}=f(x)\), where \(f(x)=-x\) for \(|x|\le1\) and \(f(x)=-\operatorname{sgn}(x)\) otherwise, \(V(x)=x^2\) satisfies (iii) and (iv), but \(\dot V(x)\le-\lambda V(x)\) fails globally for every \(\lambda>0\).
\end{remark}

For each \(R>\|\bm y_0\|\), define \(\rho_R \triangleq \inf\{ t\ge 0 : \|\bm y_t\| \ge R \}\), with \(\inf\emptyset=+\infty\). The limit \(\rho_\infty \triangleq \lim_{R\to+\infty}\rho_R\) is well defined, and by standard arguments the solution \(\bm{y}_t\) to Eq. \eqref{eq_2_2} exists on \([0,\rho_\infty)\).

In the following, we discuss the stability of \eqref{eq_2_2} separately for the two cases \(p\ge 2\) and \(p\in(0,2)\).

\section{The case \texorpdfstring{\(p\geq 2\)}{p >= 2}: a Halanay-type approach}
\label{sec_3}

In this section, we show that for the case \(p \ge 2\), the stability can be proved by a direct application of the Halanay-type inequality.

\textcolor{black}{Fix an arbitrary $R > \|\bm{y}_0\|$.} For \(t\ge0\), define
\begin{small}\begin{equation}\label{eq_3_8}
U(t)
\triangleq
\mathbb E\left[
V(\bm y_t) \cdot \mathbf1_{\{t<\rho_R\}}
\right].
\end{equation}\end{small}
We extend it to \([-\tau,0]\) by \(U(t)=U(0)\) for \(t\in[-\tau,0]\).

\subsection{Preliminary estimates and lemmas}
We first collect several preliminary estimates and lemmas that will be
used in the subsequent Halanay-type analysis.

\begin{lemma}\label{lem_3_2}
Suppose that Assumptions~\ref{ass_2_1} and \ref{ass_2_2} hold. Then, 
\begin{small}\begin{equation}\label{eq_3_9}
\mathbb E\left[
\|\bm y_t-\bm y_{\delta_t}\|^p
\cdot \mathbf1_{\{t<\rho_R\}}
\right]
\leq
K(R,p,\tau)\sup_{\delta_t\leq r\leq t}U(r), \forall t\geq 0,
\end{equation}\end{small}
where
\begin{small}\begin{equation}\label{eq_3_10}
K(R,p,\tau)
\triangleq
\frac{2^p}{c_1}
\left[
M_{f,R}^{(p)}\tau^p
+
\left(\frac{p(p-1)}{2}\right)^{p/2}
M_{g,R}^{(p)}\tau^{p/2}
\right].
\end{equation}\end{small}
\end{lemma}

\begin{proof}
Integrating \eqref{eq_2_2} over
\([\delta_t\wedge\rho_R,t\wedge\rho_R]\) yields
\begin{small}\[
\bm y_{t\wedge\rho_R}-\bm y_{\delta_t\wedge\rho_R}
=
\int_{\delta_t}^{t}
\bm{f}(\bm y_s, \bm\alpha(\bm y_{\delta_s}))
\cdot\mathbf1_{\{s<\rho_R\}}\,\mathrm ds
+
\int_{\delta_t}^{t}
\bm{g}(\bm y_s,\bm \alpha(\bm y_{\delta_s}))
\cdot\mathbf1_{\{s<\rho_R\}}\,\mathrm dB_s.
\]\end{small}
Multiplying both sides by \(\mathbf1_{\{t<\rho_R\}}\) and noting that
$
\left(
\bm y_{t\wedge\rho_R}-\bm y_{\delta_t\wedge\rho_R}
\right)\cdot\mathbf1_{\{t<\rho_R\}}
=
\left(
\bm y_t-\bm y_{\delta_t}
\right)\cdot\mathbf1_{\{t<\rho_R\}},
$
we obtain, by Hölder's inequality and
\cite[Chapter~1, Theorem~7.1]{maobook},
\begin{small}\[
\begin{aligned}
&\mathbb E\left[
\|\bm y_t-\bm y_{\delta_t}\|^p
\cdot \mathbf1_{\{t<\rho_R\}}
\right]\leq
2^{p-1}(t-\delta_t)^{p-1}
\int_{\delta_t}^{t}
\mathbb E\left[
\|\bm{f}(\bm y_s,\bm \alpha(\bm y_{\delta_s}))\|^p
\cdot\mathbf1_{\{s<\rho_R\}}
\right]\mathrm ds\\
&\quad+
2^{p-1}
\left(\frac{p(p-1)}{2}\right)^{p/2}
(t-\delta_t)^{p/2-1}
\int_{\delta_t}^{t}
\mathbb E\left[
\|\bm{g}(\bm y_s,\bm \alpha(\bm y_{\delta_s}))\|^p
\cdot \mathbf1_{\{s<\rho_R\}}
\right]\mathrm ds.
\end{aligned}
\]\end{small}
On \(\{s<\rho_R\}\), one has
\(\|\bm y_s\|\vee\|\bm y_{\delta_s}\|\leq R\). Moreover,
\(\delta_s=\delta_t\) for almost every \(s\in[\delta_t,t]\). Therefore,
by Assumption \ref{ass_2_2}(ii) and Remark \ref{rem_2_1},
\begin{small}\[
\begin{aligned}
\mathbb E\left[
\|\bm f (\bm y_s,\bm \alpha(\bm y_{\delta_s}))\|^p
\cdot\mathbf1_{\{s<\rho_R\}}
\right]
&\leq
\frac{M_{f,R}^{(p)}}{c_1}
\left[U(s)+U(\delta_s)\right],\\
\mathbb E\left[
\|\bm g (\bm y_s,\alpha (\bm y_{\delta_s}))\|^p
\cdot \mathbf1_{\{s<\rho_R\}}
\right]
&\leq
\frac{M_{g,R}^{(p)}}{c_1}
\left[U(s)+U(\delta_s)\right].
\end{aligned}
\]\end{small}
Substitution of these estimates gives
\begin{small}\[
\begin{aligned}
&\mathbb E\left[
\|\bm y_t-\bm y_{\delta_t}\|^p
\cdot \mathbf1_{\{t<\rho_R\}}
\right]\\
\leq&
\frac{2^p}{c_1}
\left[
M_{f,R}^{(p)}(t-\delta_t)^p
+
\left(\frac{p(p-1)}{2}\right)^{p/2}
M_{g,R}^{(p)}(t-\delta_t)^{p/2}
\right]
\sup_{\delta_t\leq r\leq t}U(r).
\end{aligned}
\]\end{small}
And by \(t-\delta_t\leq\tau\), we prove \eqref{eq_3_9}.
\end{proof}

\begin{lemma}\label{1406}
Suppose that Assumptions~\ref{ass_2_1} and \ref{ass_2_2} hold. Then, for every \(R>0\) and all
\(\bm x,\bm z\in\mathbb R^n\) satisfying
\(\|\bm x\|\vee\|\bm z\|\leq R\), one has
\begin{small}\begin{equation}\label{1316}
V_{\bm x}(\bm x)\bm f(\bm x,\bm \alpha(\bm z))
+\frac12\bm g^\top(\bm x,\bm \alpha(\bm z))
V_{\bm x\bm x}(\bm x)\bm g(\bm x,\bm \alpha(\bm z))
\leq
-\frac{\lambda_R}{2}V(\bm x)
+H(R,p)\|\bm x-\bm z\|^p,
\end{equation}\end{small}
where
\begin{small}$
H(R,p)
\triangleq
A_R\left(\frac{4A_R}{c_1\lambda_R}\right)^{p-1}+
B_R\left(\frac{4B_R}{c_1\lambda_R}\right)^{\frac{p-2}{2}},
A_R
\triangleq
c_{2,R}D_{f,R}
+c_{3,R}\Gamma_RD_{g,R},
B_R
\triangleq
\frac12c_{3,R}D_{g,R}^2.
$\end{small}
\end{lemma}

\begin{proof}
Denote the left-hand side of \eqref{1316} by
\(\mathcal L_{\bm z}V(\bm x)\) and set
\(d=\|\bm x-\bm z\|\). Adding and subtracting the corresponding closed-loop terms and using Assumptions~\ref{ass_2_2}, we obtain
\begin{small}\[
\begin{aligned}
\mathcal L_{\bm z}V(\bm x)
&\leq
-\lambda_RV(\bm x)
+c_{2,R}D_{f,R}\|\bm x\|^{p-1}d+
\frac12c_{3,R}\|\bm x\|^{p-2}
\left(
2\Gamma_RD_{g,R}\|\bm x\|d
+D_{g,R}^2d^2
\right)\\
&=
-\lambda_RV(\bm x)
+A_R\|\bm x\|^{p-1}d
+B_R\|\bm x\|^{p-2}d^2.
\end{aligned}
\]\end{small}
By Young's inequality,
\begin{small}$
A_R\|\bm x\|^{p-1}d
\leq
\frac{c_1\lambda_R}{4}\|\bm x\|^p
+
A_R\left(\frac{4A_R}{c_1\lambda_R}\right)^{p-1}d^p,
B_R\|\bm x\|^{p-2}d^2
\leq
\frac{c_1\lambda_R}{4}\|\bm x\|^p
+
B_R\left(\frac{4B_R}{c_1\lambda_R}\right)^{\frac{p-2}{2}}d^p.
$\end{small}
The conclusion follows from
\(V(\bm x)\geq c_1\|\bm x\|^p\).
\end{proof}

We next establish a Halanay inequality for our analysis. Baker and Buckwar \cite{BakerBuckwar2005} used such an argument for SDEs with delays. Xu, Mao, and Yin \cite{XuMaoYin2026} extended it to right-continuous functions for hybrid systems. However, since our $U(t)$ has specific jumps, we impose assumption \eqref{eq:jump} to rule out upward jumps and ensure the estimates hold across discontinuities. Lemma \ref{lem_3_1} is then established; its proof is deferred to {\bf Appendix A}.

\begin{lemma}\label{lem_3_1}
Let $\tau>0$ and $a>b\ge 0$. Let
$
\Phi:[-\tau,\infty)\longrightarrow[0,\infty)
$
be locally bounded. Assume that, for every $t\ge0$,
\begin{equation}
\limsup_{s\to t^+}\Phi(s)
\le \Phi(t)
\le \liminf_{s\to t^-}\Phi(s),
\label{eq:jump}\end{equation}
and that the upper right Dini derivative satisfies
\begin{equation}
D^+\Phi(t)
\le -a\Phi(t)
+b\sup_{-\tau\le\theta\le0}\Phi(t+\theta),
\quad t\ge0,
\label{eq:dini-ineq}
\end{equation}
where
$
D^+\Phi(t)
\triangleq\limsup_{h\to0^+}
\frac{\Phi(t+h)-\Phi(t)}{h}$.
Then there exists a unique $\mu>0$ satisfying $\mu=a-b{\mathrm e}^{\mu\tau}$,
and for $t\geq 0$,
$\Phi(t)\le {\mathrm e}^{-\mu t}\sup_{-\tau\le s\le0}\Phi(s)$.
\end{lemma}

\subsection{A Halanay-type inequality and related estimates}

\begin{lemma}\label{lem_3_3}
Let Assumptions~\ref{ass_2_1} and \ref{ass_2_2} hold. Suppose that $H(R,p)K(R,p,\tau)$ $<\frac{\lambda_R}{2}$.
Let \(\mu_R(\tau)>0\) be the unique solution of
\begin{small}\begin{equation}\label{eq_3_13}
\mu_R(\tau)+H(R,p)K(R,p,\tau){\mathrm e}^{\mu_R(\tau)\tau}=\frac{\lambda_R}{2}.
\end{equation}\end{small}
Then, for every \(t\ge0\),
$U(t)\le V(\bm y_0){\mathrm e}^{-\mu_R(\tau) t}$.
\end{lemma}

\begin{proof}
Since $V$ is continuous and hence bounded on $\|\bm{y}_t\|\leq R$, the dominated convergence theorem and the continuity of $\bm y_t$ yield
$\lim_{s\to t+}U(s)=U(t)$
and
$
\lim_{s\to t-}U(s)
=\lim_{s\to t-}
\mathbb{E}\left[
V(\bm y_s)\cdot\mathbf{1}_{\{s<\rho_R\}}
\right] =
\mathbb{E}\left[
V(\bm y_t)\cdot\mathbf{1}_{\{t\leq\rho_R\}}
\right]. $
Consequently,
$
\lim_{s\to t-}U(s)-U(t)
=
\mathbb{E}\left[
V(\bm y_t)\cdot\mathbf{1}_{\{\rho_R=t\}}
\right]\geq0.
$
Thus, $U$ is right-continuous and can only have downward jumps on the left. Next we establish a Halanay-type inequality for $U(t)$. For every $h>0$,
\begin{small}\begin{align}
\frac{U(t+h)-U(t)}{h}
={}&
\frac{1}{h}\mathbb{E}\left[
V(\bm y_{(t+h)\wedge\rho_R})
-
V(\bm y_{t\wedge\rho_R})
\right]-
\frac{1}{h}\mathbb{E}\left[
V(\bm y_{\rho_R})
\cdot\mathbf{1}_{\{t<\rho_R\leq t+h\}}
\right]
\nonumber\\
\leq{}&
\frac{1}{h}\mathbb{E}\left[
V(\bm y_{(t+h)\wedge\rho_R})
-
V(\bm y_{t\wedge\rho_R})
\right].
\label{eq_U_difference}
\end{align}\end{small} Applying It\^{o}'s formula to the stopped process gives
\begin{small}$$
V(\bm y_{(t+h)\wedge\rho_R})
-
V(\bm y_{t\wedge\rho_R})
=
\int_t^{t+h}
\mathbf{1}_{\{s<\rho_R\}}
\mathcal{L}_{\bm y_{\delta_s}}V(\bm y_s)\,{\mathrm d}s+\int_t^{t+h}\mathbf{1}_{\{s<\rho_R\}}
\cdots \,{\mathrm d}B_s,
$$\end{small}
Taking expectations and using Lemma \ref{1406},
we obtain from \eqref{eq_U_difference} that
\begin{small}\begin{align}
\frac{U(t+h)-U(t)}{h}
\leq
\frac{1}{h}\int_t^{t+h}
\left[
-\frac{\lambda_R}{2}U(s)
+
H(R,p)
\mathbb{E}\left(
\|\bm y_s-\bm y_{\delta_s}\|^p
\cdot \mathbf{1}_{\{s<\rho_R\}}
\right)
\right]\,ds.
\label{eq_U_difference_2}
\end{align}\end{small}
By \eqref{eq_3_9} and $\delta_s\geq(s-\tau)\vee0$ and $U$ has been extended to
$[-\tau,0]$, \eqref{eq_U_difference_2} yields
\begin{small}\[
\frac{U(t+h)-U(t)}{h}
\leq
\frac{1}{h}\int_t^{t+h}
\left[
-\frac{\lambda_R}{2}U(s)
+
H(R,p)K(R,p,\tau)
\sup_{s-\tau\leq r\leq s}U(r)
\right]\,{\mathrm d}s.
\]\end{small}
Letting $h\to 0+$ and using the right-continuity of $U$, we obtain
\begin{small}\[
D^+U(t)
\leq
-\frac{\lambda_R}{2}U(t)
+
H(R,p)K(R,p,\tau)
\sup_{t-\tau\leq r\leq t}U(r).
\]\end{small} By Lemma \ref{lem_3_1}, we obtain the desired result.
\end{proof}

\begin{lemma}\label{lem_3_4}
Let $R>\|\bm y_0\|$, and suppose that the assumptions of Lemma \ref{lem_3_3} hold. Then
\begin{small}\[
\mathbb{P}(\rho_R<+\infty)
\leq
\left[
1+
H(R,p)K(R,p,\tau)
\frac{\tau}{1-e^{-\mu_R(\tau)\tau}}
\right]
\frac{V(\bm y_0)}{c_1R^p}.
\]\end{small}
\end{lemma}

\begin{proof}

Applying It\^{o}’s formula to $V(\bm y_{t\wedge\rho_R})$ and using Lemma \ref{1406}, we obtain
\begin{small}\begin{align*}
\mathbb{E}V(\bm y_{t\wedge\rho_R})
\leq{}&
V(\bm y_0)
+\int_0^t
\left[
-\frac{\lambda_R}{2}U(s)
+
H(R,p)
\mathbb{E}\left(
\|\bm y_s-\bm y_{\delta_s}\|^p
\cdot \mathbf{1}_{\{s<\rho_R\}}
\right)
\right]{\rm d}s\\
\leq{}&
V(\bm y_0)
+
H(R,p)K(R,p,\tau)
\int_0^t
\sup_{\delta_s\leq r\leq s}U(r)\,{\rm d}s.
\end{align*}\end{small}
 By Lemma~\ref{lem_3_3}, for $\delta_s\leq r\leq s$,
$
U(r)
\leq
V(\bm y_0)e^{-\mu_R(\tau)r}
\leq
V(\bm y_0)e^{-\mu_R(\tau)\delta_s}.
$ Hence,
\begin{small}\begin{align*}
\mathbb{E}V(\bm y_{t\wedge\rho_R})
&\leq
V(\bm y_0)
+
H(R,p)K(R,p,\tau)V(\bm y_0)
\int_0^t e^{-\mu_R(\tau)\delta_s}\,{\mathrm d}s\\
&\leq
V(\bm y_0)
+
H(R,p)K(R,p,\tau)V(\bm y_0)
\sum_{n=0}^{\infty}
\int_{n\tau}^{(n+1)\tau}
e^{-\mu_R(\tau)n\tau}\,{\mathrm d}s\\
&=
\left[
1+
H(R,p)K(R,p,\tau)
\frac{\tau}{1-e^{-\mu_R(\tau)\tau}}
\right]V(\bm y_0).
\end{align*}\end{small} 
On the other hand,  \begin{small}$c_1R^p\mathbb{P}(\rho_R\leq t)
\leq
\mathbb{E}\left[
V(\bm y_{\rho_R})\cdot\mathbf{1}_{\{\rho_R\leq t\}}
\right]\leq
\mathbb{E}V(\bm y_{t\wedge\rho_R})$\end{small}. Substituting this into the last inequality and letting $t\to+\infty$ yields the desired assertion.
\end{proof}

\subsection{Almost sure exponential stability on a non-exit event}
In this subsection, we establish almost sure exponential stability on the non-exit event $\{\rho_R=+\infty\}.$

\begin{lemma}\label{lem_3_5}
Suppose that the assumptions of Lemma \ref{lem_3_3} hold. Then
\begin{small}\begin{equation}\label{0237}
\limsup_{t\to+\infty}
\frac1t\log\|\bm y_t\|
\le
-\frac{\mu_R(\tau)}{p}
\quad
\text{{\rm a.s.} on }\{\rho_R=+\infty\}.
\end{equation}\end{small}
\end{lemma}

\begin{proof}
By Lemma~\ref{lem_3_3} and $V(\bm x)\geq c_1\|\bm x\|^p$,
\begin{small}\begin{equation}
\mathbb E\left[\|\bm y_t\|^p
\cdot
\mathbf 1_{\{t<\rho_R\}}\right]
\leq\frac{V(\bm y_0)}{c_1}e^{-\mu_R(\tau)t}.
\label{eq_path_moment}
\end{equation}\end{small}
For $a\in(0,1]$, define $\zeta_a\triangleq\inf\{t\geq0:\|\bm y_t\|\geq a\}$.
Applying It\^{o}'s formula to
$\|\bm y_{t\wedge\zeta_a}\|^p$ on $[0,\tau]$, 
taking expectation and using $\|\bm{y}_0\|\vee\|\bm y_s\|<a\leq 1$ on $\{s<\zeta_a\}$, 
the local Lipschitz conditions imply that there exists a constant $C>0$ such that
\begin{small}\begin{align*}
&\mathbb E\|\bm y_{t\wedge\zeta_a}\|^p
=
\|\bm y_0\|^p
+\int_0^t
\mathbb E\bigl[
\bigl(p\|\bm y_s\|^{p-2}
\bm y_s^ \top\bm f(\bm y_s,\bm\alpha(\bm y_0))
+\frac p2\|\bm y_s\|^{p-2}
\|\bm g(\bm y_s,\bm\alpha(\bm y_0))\|^2\\
&+\frac{p(p-2)}2\|\bm y_s\|^{p-4}
\left(\bm y_s^\top \bm g(\bm y_s,\bm\alpha(\bm y_0))\right)^2
\bigr)\cdot \mathbf 1_{\{s<\zeta_a\}}
\bigr]\mathrm ds\\
&\leq
\|\bm y_0\|^p
+C\int_0^t
\mathbb E\left[
\left(\|\bm y_s\|^p+\|\bm y_0\|^p\right)
\cdot\mathbf 1_{\{s<\zeta_a\}}
\right]\mathrm ds\leq
(1+Ct)\|\bm y_0\|^p
+C\int_0^t
\mathbb E\|\bm y_{s\wedge\zeta_a}\|^p\,\mathrm ds.
\end{align*}\end{small}
The Gronwall's inequality gives
$
\mathbb E\|\bm y_{\tau\wedge\zeta_a}\|^p
\leq (1+C\tau)e^{C\tau}\|\bm y_0\|^p.
$
Consequently,
\begin{small}\begin{equation}\label{1641}
\mathbb P\left(\sup_{0\leq t\leq\tau}\|\bm y_t\|\geq a\right)
=\mathbb P(\zeta_a\leq\tau)
\leq\frac{1}{a^p}\mathbb E\|\bm y_{\tau\wedge\zeta_a}\|^p
\leq\frac{(1+C\tau)e^{C\tau}}{a^p}\|\bm y_0\|^p.
\end{equation}\end{small} 
Shifting \eqref{1641} to $[k\tau,(k+1)\tau]$, multiplying $\bm{1}_{k\tau<\rho_R}$, choosing \begin{small}$a_\epsilon=\exp\left\{-\frac{\mu_R(\tau)-\varepsilon}{p}\tau\right\}$\end{small} for $\epsilon>0$ sufficiently small, taking expectation and using \eqref{eq_path_moment}  gives
\begin{small}\begin{align*}&\sum_{k=1}^{+\infty}\mathbb P\left(\sup_{t\in[k\tau,(k+1)\tau]}\|\bm{y}_t\|\geq a_\epsilon^k;k\tau<\rho_R\right)
\leq (1+C\tau)e^{C\tau}\sum_{k=1}^{+\infty}\dfrac{\mathbb{E}\left[\|\bm{y}_{k\tau}\|^p;k\tau<\rho_R\right]}{a_\epsilon^{kp}}\\
&\leq \frac{(1+C\tau){\rm e}^{c\tau}V(\bm{y_0})}{c_1}\sum_{k=1}^{+\infty}{\rm e}^{-\epsilon k\tau}<+\infty.
\end{align*}
\end{small} By the Borel--Cantelli lemma, 
 \begin{small}$
\limsup_{t\to\infty}\dfrac{\log\|\bm{y}_t\|}{t}\leq-\frac{\mu_R(\tau)-\varepsilon}{p}
\quad\text{a.s. on }\{\rho_R=\infty\}.
$\end{small}
 Letting $\epsilon\to 0+$ yields \eqref{0237}.
\end{proof}

\subsection{Parameter selection and main result}
By suitably choosing the parameters, we establish the main result of
this section. 
\begin{theorem}\label{thm_3_1}
Suppose that the assumptions of Lemma \ref{lem_3_3} hold.  Given
\(\varepsilon\in(0,1)\) and 
\(\bm y_0\in\mathbb R^n\backslash\{\bm{0}\}\), there exists $\tau^*=\tau^*(V(\bm{y}_0),\varepsilon)>0$ such that for $\tau\in(0,\tau^*)$,
\begin{small}\begin{equation}\label{eq_3_30}
\mathbb P\left(
\{\rho_\infty=+\infty\}
\cap
\left\{
\limsup_{t\to+\infty}
\frac1t\log\|\bm y_t\|<0
\right\}
\right)
\ge
1-\varepsilon.
\end{equation}\end{small}
\end{theorem}

\begin{proof}
 Fix
$\varepsilon\in(0,1)$ and
$\bm y_0\in\mathbb R^n\setminus\{\bm 0\}$. We first choose $R$
such that
\begin{small}\begin{equation}\label{eq_3_21}
R>
\max\left\{
\|\bm y_0\|,
\left(
\frac{2V(\bm y_0)}{c_1\varepsilon}
\right)^{1/p}
\right\}.
\end{equation}\end{small}
We then choose $\tau^*>0$ sufficiently small such that, for every
$\tau\in(0,\tau^*)$, \begin{small}\begin{equation}\label{1818}
H(R,p)K(R,p,\tau)<\frac{\lambda_R}{2},
\quad
1+
H(R,p)K(R,p,\tau)
\frac{\tau}{1-\mathrm e^{-\mu_R(\tau)\tau}}
<2. \end{equation}\end{small}
We verify that such a $\tau^*$ exists. Since
$
H(R,p)K(R,p,\tau)=O(\tau^{p/2})\to0
$ as $\tau\to 0+$.
Thus, the first inequality holds for all sufficiently small
$\tau>0$. Moreover, it follows from \eqref{eq_3_13} that
$
\mu_R(\tau)\longrightarrow\frac{\lambda_R}{2}\text{as }\tau\to0^+.
$
Consequently,
$
\frac{\tau}
{1-\mathrm e^{-\mu_R(\tau)\tau}}
\to
\frac{2}{\lambda_R},
$
and hence
$
H(R,p)K(R,p,\tau)
\frac{\tau}{1-\mathrm e^{-\mu_R(\tau)\tau}}
=
O(\tau^{p/2})
\to 0.
$
Therefore, the second inequality also holds for all
sufficiently small $\tau>0$. By Lemma~\ref{lem_3_4}, \eqref{eq_3_21} and \eqref{1818} give
\begin{small}$$
\mathbb P(\rho_R<+\infty)
\leq
\left[
1+
H(R,p)K(R,p,\tau)
\frac{\tau}{1-\mathrm e^{-\mu_R(\tau)\tau}}
\right]
\frac{V(\bm y_0)}{c_1R^p}<
\frac{2V(\bm y_0)}{c_1R^p}
<\varepsilon.
$$\end{small} Therefore, $\mathbb P(\rho_R=+\infty)\geq1-\varepsilon$.
By Lemma~\ref{lem_3_5}, almost surely on $\{\rho_R=+\infty\}$,
$
\limsup_{t\to+\infty}\frac1t\log\|\bm y_t\|
\leq-\frac{\mu_R(\tau)}p<0.
$
Since
$
\{\rho_R=+\infty\}\subset\{\rho_\infty=+\infty\},
$
we have
\begin{small}\[
\mathbb P\left(
\{\rho_\infty=+\infty\}
\cap
\left\{
\limsup_{t\to+\infty}
\frac1t\log\|\bm y_t\|<0
\right\}
\right)\geq
\mathbb P(\rho_R=+\infty)
\geq1-\varepsilon,
\]\end{small}
which proves \eqref{eq_3_30}.
\end{proof}

\section{The case \texorpdfstring{\(0<p<2\)}{0 < p < 2}: a comparison approach via an auxiliary system}
\label{sec_4}
In this section, for the case \(p\in(0,2)\), we use the comparison method developed in \cite{Mao2013,mao2015almost}.

\subsection{Preliminary estimates}

\begin{lemma}\label{lem_4_1}
Suppose that Assumption \ref{ass_2_1} and \ref{ass_2_2} hold. For every $R>0$ and
$\bm x,\bm y\in\mathbb R^n$ satisfying
$
\|\bm x\|\vee\|\bm y\|\leq R,$
the following estimates hold.

If $0<p\leq1$, then
\begin{small}\begin{equation}\label{eq_4_1}
|V(\bm x)-V(\bm y)|
\leq
\frac{2^{1-p}c_{2,R}}{p}\|\bm x-\bm y\|^p.
\end{equation}\end{small}
If $1<p<2$, then
\begin{small}\begin{equation}\label{eq_4_1b}
|V(\bm x)-V(\bm y)|
\leq
c_{2,R}
\left(
\|\bm x\|^{p-1}+\|\bm y\|^{p-1}
\right)
\|\bm x-\bm y\|.
\end{equation}\end{small}
In this case, for every $\varkappa\in(0,1)$,
\begin{small}\begin{equation}\label{eq_4_1c}
V(\bm x)
\leq
\frac{1+\varkappa}{1-\varkappa}V(\bm y)
+
\frac{
2(p-1)^{p-1}c_{2,R}^p
}{
p^p\varkappa^{p-1}(1-\varkappa)c_1^{p-1}
}
\|\bm x-\bm y\|^p.
\end{equation}\end{small}
\end{lemma}

\begin{proof}
If $\bm x=\bm y$, the assertion is trivial. Suppose that
$\bm x\neq\bm y$. By the integral form of the mean value theorem and
Assumption \ref{ass_2_2}(iv),
\begin{small}$$
|V(\bm x)-V(\bm y)|
\leq
c_{2,R}\|\bm x-\bm y\|
\int_0^1
\|\bm y+\theta(\bm x-\bm y)\|^{p-1}\,\mathrm d\theta.
$$\end{small}
If the line segment joining $\bm x$ and $\bm y$ passes through the
origin, the above inequality follows by splitting the segment at the
origin and using a limiting argument. 
We first consider $0<p<1$. Since $p-1<0$,
\begin{small}\begin{align*}
&\int_0^1
\|\bm y+\theta(\bm x-\bm y)\|^{p-1}\,\mathrm d\theta\leq
\int_0^1
\left|
\left(\bm y+\theta(\bm x-\bm y)\right)^\top
\frac{\bm x-\bm y}{\|\bm x-\bm y\|}
\right|^{p-1}\,\mathrm d\theta
=
\|\bm x-\bm y\|^{p-1}\cdot\\&
\int_0^1
\left|
\frac{\bm y^\top(\bm x-\bm y)}
{\|\bm x-\bm y\|^2}
+\theta
\right|^{p-1}\,\mathrm d\theta\leq
\|\bm x-\bm y\|^{p-1}
\sup_{a\in\mathbb R}
\int_a^{a+1}|u|^{p-1}\,\mathrm du=
\frac{2^{1-p}}{p}\|\bm x-\bm y\|^{p-1}.
\end{align*}\end{small}
Here the last equality follows because $|u|^{p-1}$ is even and
decreasing with respect to $|u|$. Combining this estimate with the
preceding inequality proves \eqref{eq_4_1} for $0<p<1$.

If $p=1$, Assumption~\ref{ass_2_2} directly gives \eqref{eq_4_1}.

Now we consider $1<p<2$. Since $p-1\in(0,1)$, for $\theta\in[0,1]$,
\begin{small}$
\|\bm y+\theta(\bm x-\bm y)\|^{p-1}
\leq
\|\bm x\|^{p-1}+\|\bm y\|^{p-1}$.\end{small}
Together with the integral mean-value estimate, this proves
\eqref{eq_4_1b}. Finally, Young's inequality and
$V(\bm z)\geq c_1\|\bm z\|^p$ from
Assumption~\ref{ass_2_2} give, for $\bm z=\bm x,\bm y$,
\begin{small}$$
\begin{aligned}
c_{2,R}\|\bm z\|^{p-1}\|\bm x-\bm y\|
&\leq
\varkappa c_1\|\bm z\|^p
+
\frac{(p-1)^{p-1}c_{2,R}^p}
{p^p\varkappa^{p-1}c_1^{p-1}}
\|\bm x-\bm y\|^p\\
&\leq
\varkappa V(\bm z)
+
\frac{(p-1)^{p-1}c_{2,R}^p}
{p^p\varkappa^{p-1}c_1^{p-1}}
\|\bm x-\bm y\|^p.
\end{aligned}
$$\end{small}
Applying this estimate to the right-hand side of
\eqref{eq_4_1b} and rearranging proves \eqref{eq_4_1c}.
\end{proof}

\subsection{Finite-time comparison estimates via an auxiliary system}
We next establish finite-time estimates for the sampled-feedback systems by comparing them with suitable auxiliary
systems. Let $T>0$ be such that $T/\tau\in\mathbb N$. We consider system \eqref{eq_2_1} and \eqref{eq_2_2} on $[0,T]$
with the same deterministic initial value
$\bm{x}_0=\bm{y}_0$.

Let \(R>0\) be a radius to be
specified later, and define
\begin{small}\begin{equation}\label{eq_4_5}
\xi_1
\triangleq
\inf\{t \ge 0:\|\bm{x}_t\|\ge R\},
\quad
\xi_2
\triangleq
\inf\{t\ge 0:\|\bm{y}_t\|\ge R\}.
\end{equation}\end{small}
The following lemmas provide several estimates for the auxiliary systems.
\begin{lemma}\label{lem_4_2}
Suppose that Assumptions~\ref{ass_2_1} and \ref{ass_2_2} hold with
$0<p<2$, and let $\bm x_t$ be the solution of \eqref{eq_2_1}.
Let $T>0$ and $\|\bm x_0\|<R$. Then, for every
$t\in[0,T]$,
\begin{small}\begin{equation}\label{eq_4_2}
\mathbb E\left[
V(\bm x_t)
\cdot
\mathbf 1_{\{t\leq\xi_1\}}
\right]
\leq
\mathrm e^{-\lambda_Rt}V(\bm x_0).
\end{equation}\end{small}
\end{lemma}
\begin{proof}
Clearly, \eqref{eq_4_2} holds if $\bm x_0=\bm 0$. Suppose that $\bm x_0\neq\bm 0$, applying \cite[Chapter 4, Lemma 3.2]{maobook}, we know that the trajectory $\bm{x}_t$ will not approach zero in finite time.
By It\^{o}'s formula we obtain
\begin{small}\begin{align*}
\mathrm e^{\lambda_R(t\wedge\xi_1)}
V(\bm x_{t\wedge\xi_1})
={}&
V(\bm x_0)
+
\int_0^t
\mathrm e^{\lambda_Rs}
\left[
\lambda_RV(\bm x_s)
+\mathcal LV(\bm x_s)
\right]
\cdot
\mathbf 1_{\{s\leq\xi_1\}}\,\mathrm ds+
\int_0^t \cdots\mathrm dB_s.
\end{align*}\end{small}
By Assumption~\ref{ass_2_2},  $\lambda_RV(\bm x_s)
+\mathcal LV(\bm x_s)\leq 0$ on the set  $\{s\leq\xi_1\}$. Taking expectations yields
$
\mathbb E\left[
\mathrm e^{\lambda_R(t\wedge\xi_1)}
V(\bm x_{t\wedge\xi_1})
\right]
\leq
V(\bm x_0)$. Using $\mathbb E\left[
\mathrm e^{\lambda_R(t\wedge\xi_1)}
V(\bm x_{t\wedge\xi_1})
\right]\geq \mathbb E\left[
{\rm e}^{\lambda_R t}V(\bm x_t)
\cdot\mathbf 1_{\{t\leq\xi_1\}}
\right]$ completes the proof.
\end{proof}

\begin{lemma}\label{lem_4_3}
Suppose that Assumption~\ref{ass_2_1} holds, and let \(\bm y_t\) be the
solution of \eqref{eq_2_2}. Let \(T>0\) and suppose that
\(\|\bm y_0\|<R\). Then, for every \(t\in[0,T]\),
\begin{small}\begin{equation}\label{eq_4_8}
\mathbb E\|\bm y_{t\wedge\xi_2}\|^2
\leq
\|\bm y_0\|^2\mathrm e^{\Theta_Rt},
\end{equation}\end{small}
where
\begin{small}$
\Theta_R
\triangleq
3\sqrt{M_{f,R}^{(2)}}+2M_{g,R}^{(2)}.$\end{small}
\end{lemma}

\begin{proof}
Applying It\^{o}'s formula to
\(\|\bm y_{t\wedge\xi_2}\|^2\) and taking expectations gives
\begin{small}\begin{equation}\label{eq_4_9}
\begin{aligned}
\mathbb E\|\bm y_{t\wedge\xi_2}\|^2
={}&
\mathbb E\|\bm y_{\delta_t\wedge\xi_2}\|^2+
\int_{\delta_t}^t
\mathbb E\left[
\left(
2\bm y_s^\top\bm f(\bm y_s,\bm \alpha(\bm y_{\delta_s}))
+
\|\bm g(\bm y_s,\bm \alpha(\bm y_{\delta_s}))\|^2
\right)
\cdot\mathbf1_{\{s\leq\xi_2\}}
\right]\mathrm ds.
\end{aligned}
\end{equation}\end{small}
On \(\{s\leq\xi_2\}\), one has
\(\|\bm y_s\|\vee\|\bm y_{\delta_s}\|\leq R\). Hence,
Remark~\ref{rem_2_1} and Young's inequality give
\begin{small}\[
\begin{aligned}
2\bm y_s^\top\bm f(\bm y_s,\bm \alpha(\bm y_{\delta_s}))
\leq
\sqrt{M_{f,R}^{(2)}}
\left(
2\|\bm y_s\|^2+\|\bm y_{\delta_s}\|^2
\right),
\|\bm g(\bm y_s,\bm \alpha(\bm y_{\delta_s}))\|^2
\leq
M_{g,R}^{(2)}
\left(
\|\bm y_s\|^2+\|\bm y_{\delta_s}\|^2
\right).
\end{aligned}
\]\end{small}
Since \(\delta_s=\delta_t\) for almost every
\(s\in[\delta_t,t]\), it follows from \eqref{eq_4_9} that
\begin{small}
\[
\mathbb E\|\bm y_{t\wedge\xi_2}\|^2
\leq
[1+(\sqrt{M_{f,R}^{(2)}}+M_{g,R}^{(2)})(t-\delta_t)]
\mathbb E\|\bm y_{\delta_t\wedge\xi_2}\|^2+
(2\sqrt{M_{f,R}^{(2)}}+M_{g,R}^{(2)})\int_{\delta_t}^t
\mathbb E\|\bm y_{s\wedge\xi_2}\|^2\,\mathrm ds.
\]
\end{small}
Therefore, Gronwall's inequality and \(1+\beta\leq\mathrm e^\beta\)
give
\begin{small}\[
\mathbb E\|\bm y_{t\wedge\xi_2}\|^2
\leq
\mathbb E\|\bm y_{\delta_t\wedge\xi_2}\|^2
\mathrm e^{\Theta_R(t-\delta_t)}.
\]\end{small}
Iterating this estimate over the successive sampling intervals from
\(0\) to \(t\) proves \eqref{eq_4_8}.
\end{proof}

\begin{lemma}\label{lem_4_4}
Suppose that Assumption~\ref{ass_2_1} holds, and let \(\bm y_t\) be the
solution of \eqref{eq_2_2}. Let \(T>0\) and suppose that
\(\|\bm y_0\|<R\). Then, for every \(t\in[0,T]\),
\begin{small}\begin{equation}\label{eq_4_13}
\mathbb E\left\|
\bm y_{t\wedge\xi_2}-\bm y_{\delta_t\wedge\xi_2}
\right\|^2
\leq
4\tau\left(
\tau M_{f,R}^{(2)}+M_{g,R}^{(2)}
\right)
\|\bm y_0\|^2\mathrm e^{\Theta_Rt}.
\end{equation}\end{small}
\end{lemma}

\begin{proof}
By the Cauchy--Schwarz inequality and It\^o's isometry,
\begin{small}\[
\begin{aligned}
&\mathbb E\left\|
\bm y_{t\wedge\xi_2}-\bm y_{\delta_t\wedge\xi_2}
\right\|^2\\
\leq&
2(t-\delta_t)
\int_{\delta_t}^{t}
\mathbb E\left[
\|\bm f(\bm y_s,\bm \alpha(\bm y_{\delta_s}))\|^2
\cdot\mathbf1_{\{s\leq\xi_2\}}
\right]\mathrm ds
+2\int_{\delta_t}^{t}
\mathbb E\left[
\|\bm g(\bm y_s,\bm \alpha(\bm y_{\delta_s}))\|^2
\cdot
\mathbf1_{\{s\leq\xi_2\}}
\right]\mathrm ds \\
\leq &
2\left[
(t-\delta_t)M_{f,R}^{(2)}+M_{g,R}^{(2)}
\right]
\int_{\delta_t}^{t}
\left[
\mathbb E\|\bm y_{s\wedge\xi_2}\|^2
+
\mathbb E\|\bm y_{\delta_s\wedge\xi_2}\|^2
\right]\mathrm ds.
\end{aligned}
\]\end{small}
Since \(\delta_s=\delta_t\) for almost every
\(s\in[\delta_t,t]\), Lemma~\ref{lem_4_3} yields
$\mathbb E\|\bm y_{s\wedge\xi_2}\|^2
\vee
\mathbb E\|\bm y_{\delta_s\wedge\xi_2}\|^2
\leq
\|\bm y_0\|^2\mathrm e^{\Theta_Rt}.$
Using \(t-\delta_t\leq\tau\) proves \eqref{eq_4_13}.
\end{proof}

We next compare the two processes up to the stopping time
$\xi_1\wedge\xi_2$.

\begin{lemma}\label{lem_4_5}
Suppose that Assumption~\ref{ass_2_1} holds. Let \(\bm x_t\) and
\(\bm y_t\) be the solutions of \eqref{eq_2_1} and \eqref{eq_2_2},
respectively, with \(\bm x_0=\bm y_0\) and \(\|\bm y_0\|<R\).
Then, for every \(t\in[0,T]\),
\begin{small}\begin{equation}\label{eq_4_14}
\mathbb E\left\|
\bm x_{t\wedge\xi_1\wedge\xi_2}
-\bm y_{t\wedge\xi_1\wedge\xi_2}
\right\|^2
\leq
K(R,\tau,t)\|\bm y_0\|^2,
\end{equation}\end{small}
where
\begin{small}\begin{equation}\label{eq_4_15}
K(R,\tau,t)
\triangleq
4\tau
\left(
\tau M_{f,R}^{(2)}+M_{g,R}^{(2)}
\right)
\left( D_{f,R}+2D_{g,R}^2 \right)
\frac{\mathrm e^{\Theta_Rt}-1}{\Theta_R}
\mathrm e^{\left( 2\Lambda_{f,R}+D_{f,R}+2\Lambda_{g,R}^2\right)t}.\end{equation}\end{small}
\end{lemma}

\begin{proof}
Set
$
\bm e_s\triangleq\bm x_s-\bm y_s,
\bm h_s\triangleq\bm y_s-\bm y_{\delta_s}.
$
Applying It\^o's formula and using \(\bm x_0=\bm y_0\), we obtain
\begin{small}\[
\begin{aligned}
\mathbb E\|\bm e_{t\wedge\xi_1\wedge\xi_2}\|^2
={}&
\int_0^t
\mathbb E\left[
2\bm e_s^\top
\left(
\bm f(\bm x_s,\bm \alpha(\bm x_s))
-\bm f(\bm y_s,\bm \alpha(\bm y_{\delta_s}))
\right)
\cdot\mathbf1_{\{s\leq\xi_1\wedge\xi_2\}}
\right]\mathrm ds\\
&+
\int_0^t
\mathbb E\left[
\left\|
\bm g(\bm x_s,\bm \alpha(\bm x_s))
-\bm g(\bm y_s,\bm \alpha(\bm y_{\delta_s}))
\right\|^2
\cdot\mathbf1_{\{s\leq\xi_1\wedge\xi_2\}}
\right]\mathrm ds.
\end{aligned}
\]\end{small}
On \(\{s\leq\xi_1\wedge\xi_2\}\), by Remark~\ref{rem_2_1}, we obtain
\begin{small}\[
\begin{aligned}
2\bm e_s^\top
\left(
\bm f(\bm x_s,\bm \alpha(\bm x_s))
-\bm f(\bm y_s,\bm \alpha(\bm y_{\delta_s}))
\right)
&\leq
(2\Lambda_{f,R}+D_{f,R})\|\bm e_s\|^2
+D_{f,R}\|\bm h_s\|^2,\\
\left\|
\bm g(\bm x_s,\bm \alpha(\bm x_s))
-\bm g(\bm y_s,\bm \alpha(\bm y_{\delta_s}))
\right\|^2
&\leq
2\Lambda_{g,R}^2\|\bm e_s\|^2
+2D_{g,R}^2\|\bm h_s\|^2.
\end{aligned}
\]\end{small}
By Lemma~\ref{lem_4_4}, it follows that
\begin{small}\[
\begin{aligned}
&\mathbb E\|\bm e_{t\wedge\xi_1\wedge\xi_2}\|^2\\
\leq&
\left( 2\Lambda_{f,R}+D_{f,R}+2\Lambda_{g,R}^2\right)\int_0^t
\mathbb E\|\bm e_{s\wedge\xi_1\wedge\xi_2}\|^2\,\mathrm ds+
\left( D_{f,R}+2D_{g,R}^2 \right)
\int_0^t
\mathbb E\left\|
\bm y_{s\wedge\xi_2}
-\bm y_{\delta_s\wedge\xi_2}
\right\|^2\,\mathrm ds\\
\leq&
\left( 2\Lambda_{f,R}+D_{f,R}+2\Lambda_{g,R}^2\right)\int_0^t
\cdots \mathrm ds+
4\tau
\left(
\tau M_{f,R}^{(2)}+M_{g,R}^{(2)}
\right)
\left( D_{f,R}+2D_{g,R}^2 \right)
\frac{\mathrm e^{\Theta_Rt}-1}{\Theta_R}
\|\bm y_0\|^2.
\end{aligned}
\]\end{small}
Gronwall's inequality yields
$
\mathbb E\|\bm e_{t\wedge\xi_1\wedge\xi_2}\|^2
\leq
K(R,\tau,t)\|\bm y_0\|^2,
$
which proves \eqref{eq_4_14}.
\end{proof}

As a consequence of Lemmas \ref{lem_4_2} and \ref{lem_4_5}, we obtain the following estimates for the Lyapunov function. 
\begin{lemma}\label{lem_4_6}
Suppose that Assumptions~\ref{ass_2_1} and \ref{ass_2_2} hold with \(p\in(0,2)\). Let \(\bm x_t\) and \(\bm y_t\) be
the solutions of \eqref{eq_2_1} and \eqref{eq_2_2}, respectively, with
\(\bm x_0=\bm y_0\) and \(\|\bm y_0\|<R\).

If \(0<p\leq1\), then, for every \(t\in[0,T]\),
\begin{small}\begin{equation}\label{eq_4_17}
\mathbb E\left[
V(\bm y_t)
\cdot \mathbf1_{\{t\leq\xi_1\wedge\xi_2\}}
\right]
\leq
\left[
\mathrm e^{-\lambda_Rt}
+
\frac{2^{1-p}c_{2,R}}{pc_1}
K(R,\tau,t)^{p/2}
\right]V(\bm y_0),
\end{equation}\end{small}
\begin{small}\begin{equation}\label{eq_4_18}
\mathbb E\left[
V(\bm y_{t\wedge\xi_1\wedge\xi_2})
\right]
\leq
\left[
1+
\frac{2^{1-p}c_{2,R}}{pc_1}
K(R,\tau,t)^{p/2}
\right]V(\bm y_0).
\end{equation}\end{small}

If \(1<p<2\), then, for every \(\varkappa\in(0,1)\) and \(t\in[0,T]\),
\begin{small}\begin{equation}\label{eq_4_17b}
\mathbb E\left[
V(\bm y_t)
\cdot\mathbf1_{\{t\leq\xi_1\wedge\xi_2\}}
\right]
\leq
\bigg[
\frac{1+\varkappa}{1-\varkappa}\mathrm e^{-\lambda_Rt}
+
\frac{
2(p-1)^{p-1}c_{2,R}^{p}
}{
p^p\varkappa^{p-1}(1-\varkappa)c_1^p
}
K(R,\tau,t)^{p/2}
\bigg]V(\bm y_0),
\end{equation}\end{small}
\begin{small}\begin{equation}\label{eq_4_18b}
\mathbb E\left[
V(\bm y_{t\wedge\xi_1\wedge\xi_2})
\right]
\leq
\left[
\frac{1+\varkappa}{1-\varkappa}
+
\frac{
2(p-1)^{p-1}c_{2,R}^{p}
}{
p^p\varkappa^{p-1}(1-\varkappa)c_1^p
}
K(R,\tau,t)^{p/2}
\right]V(\bm y_0).
\end{equation}\end{small}
\end{lemma}

\begin{proof}
Set \(\sigma\triangleq\xi_1\wedge\xi_2\). Suppose first that
\(0<p\leq1\). On \(\{t\leq\sigma\}\), Lemma~\ref{lem_4_1} gives
\begin{small}\[
\mathbb E\left[
V(\bm y_t)\cdot\mathbf1_{\{t\leq\sigma\}}
\right]
\leq
\mathbb E\left[
V(\bm x_t)\cdot\mathbf1_{\{t\leq\sigma\}}
\right]
+
\frac{2^{1-p}c_{2,R}}{p}
\mathbb E\left[
\|\bm y_t-\bm x_t\|^p
\cdot\mathbf1_{\{t\leq\sigma\}}
\right].
\]\end{small}
Since \(\{t\leq\sigma\}\subset\{t\leq\xi_1\}\),
Lemma~\ref{lem_4_2} implies
$
\mathbb E\left[
V(\bm x_t)\cdot\mathbf1_{\{t\leq\sigma\}}\right]
\leq
\mathrm e^{-\lambda_Rt}V(\bm x_0).
$
Moreover, Hölder's inequality and Lemma~\ref{lem_4_5} yield
\begin{small}\[
\begin{aligned}
\mathbb E\left[
\|\bm y_t-\bm x_t\|^p
\mathbf1_{\{t\leq\sigma\}}
\right]
\leq
\left(
\mathbb E\left\|
\bm y_{t\wedge\sigma}-\bm x_{t\wedge\sigma}
\right\|^2
\right)^{p/2}\leq
K(R,\tau,t)^{p/2}\|\bm y_0\|^p\leq
\frac{K(R,\tau,t)^{p/2}V(\bm y_0)}{c_1}.
\end{aligned}
\]\end{small}
Since \(\bm x_0=\bm y_0\), these estimates prove
\eqref{eq_4_17}. Moreover, Lemma~\ref{lem_4_1}, the stopped It\^o
estimate in the proof of Lemma~\ref{lem_4_2}, and
Lemma~\ref{lem_4_5} give
\begin{small}\[
\begin{aligned}
\mathbb EV(\bm y_{t\wedge\sigma})
&\leq
\mathbb EV(\bm x_{t\wedge\sigma})
+
\frac{2^{1-p}c_{2,R}}{p}
\mathbb E\|
\bm y_{t\wedge\sigma}-\bm x_{t\wedge\sigma}
\|^p\leq
V(\bm x_0)
+
\frac{2^{1-p}c_{2,R}}{pc_1}
K(R,\tau,t)^{p/2}V(\bm y_0),
\end{aligned}
\]\end{small}
which proves \eqref{eq_4_18}.

For \(1<p<2\), applying \eqref{eq_4_1c} in place of
\eqref{eq_4_1}, together with the same stopped-process and difference
estimates, proves \eqref{eq_4_17b} and \eqref{eq_4_18b}.
\end{proof}

We finally estimate the probability that at least one of the two
processes leaves the ball of radius $R$ before time $T$.
\begin{lemma}\label{lem_4_7}
Under the same assumptions of Lemma \ref{lem_4_6}. Let $\bm x_t$ and $\bm y_t$ be the
solutions of \eqref{eq_2_1} and \eqref{eq_2_2}, respectively, with
$\bm x_0=\bm y_0$ and $\|\bm y_0\|<R$. Then
\begin{small}\begin{equation}\label{eq_4_19}
\mathbb P(\xi_1\leq T)
\leq
\frac{1}{c_1R^p}V(\bm y_0).
\end{equation}\end{small}
If \(0<p\leq1\), then
\begin{small}\begin{equation}\label{eq_4_20}
\mathbb P(\xi_1\wedge\xi_2\leq T)
\leq
\left[
\frac{2}{c_1R^p}
+
\frac{2^{1-p}c_{2,R}}{pc_1^2R^p}
K(R,\tau,T)^{p/2}
\right]V(\bm y_0).
\end{equation}\end{small}
If \(1<p<2\), then, for every \(\varkappa\in(0,1)\),
\begin{small}\begin{equation}\label{eq_4_20b}
\begin{aligned}
\mathbb P(\xi_1\wedge\xi_2\leq T)
\leq
\bigg[
&\frac{2}{(1-\varkappa)c_1R^p}
+
\frac{
2(p-1)^{p-1}c_{2,R}^{p}
}{
p^p\varkappa^{p-1}(1-\varkappa)c_1^{p+1}R^p
}
K(R,\tau,T)^{p/2}
\bigg]V(\bm y_0).
\end{aligned}
\end{equation}\end{small}
\end{lemma}

\begin{proof}
On \(\{\xi_1\leq T\}\), the continuity of the sample paths and the
definition of \(\xi_1\) imply that \(\|\bm x_{\xi_1}\|=R\). Hence,
by Assumption \ref{ass_2_2}(ii),
$
c_1R^p\cdot\mathbf 1_{\{\xi_1\leq T\}}
\leq
V(\bm x_{\xi_1})\cdot\mathbf 1_{\{\xi_1\leq T\}}
\leq
V(\bm x_{T\wedge\xi_1}).
$
Taking expectations and using the stopped Lyapunov estimate in
Lemma~\ref{lem_4_2}, we obtain
$c_1R^p\mathbb P(\xi_1\leq T)
\leq
\mathbb EV(\bm x_{T\wedge\xi_1})
\leq
V(\bm x_0)=V(\bm y_0)$,
which proves \eqref{eq_4_19}.

Next,
$
\{\xi_1\wedge\xi_2\leq T\}
\subset
\{\xi_1\leq T\}
\cup
\{\xi_2\leq T,\ \xi_2<\xi_1\}.
$
On \(\{\xi_2\leq T,\ \xi_2<\xi_1\}\), one has
\(\|\bm y_{\xi_2}\|=R\) and
\(\xi_1\wedge\xi_2=\xi_2\). Therefore,
\begin{small}\[
c_1R^p
\cdot
\mathbf 1_{\{\xi_2\leq T,\ \xi_2<\xi_1\}}
\leq
V(\bm y_{\xi_2})
\cdot\mathbf 1_{\{\xi_2\leq T,\ \xi_2<\xi_1\}}
\leq
V(\bm y_{T\wedge\xi_1\wedge\xi_2}).
\]\end{small}
If \(0<p\leq1\), taking expectations and using \eqref{eq_4_18} with
\(t=T\), we obtain
\begin{small}\[
\begin{aligned}
\mathbb P(\xi_2\leq T,\ \xi_2<\xi_1)
&\leq
\frac{1}{c_1R^p}
\mathbb EV(\bm y_{T\wedge\xi_1\wedge\xi_2})\\
&\leq
\frac{1}{c_1R^p}
\left[
1+
\frac{2^{1-p}c_{2,R}}{pc_1}
K(R,\tau,T)^{p/2}
\right]V(\bm y_0).
\end{aligned}
\]\end{small}
Using \begin{small}$\mathbb{P}(\xi_1\wedge\xi_2\leq T)\leq \mathbb{P}(\xi_1\leq T)+\mathbb{P}(\xi_2\leq T,\ \xi_2<\xi_1)$\end{small} and \eqref{eq_4_19} proves \eqref{eq_4_20}.

If \(1<p<2\), \eqref{eq_4_18b} similarly gives
\begin{small}\[
\begin{aligned}
\mathbb P(\xi_2\leq T,\ \xi_2<\xi_1)
\leq
\frac{1}{c_1R^p}
\left[
\frac{1+\varkappa}{1-\varkappa}
+
\frac{
2(p-1)^{p-1}c_{2,R}^{p}
}{
p^p\varkappa^{p-1}(1-\varkappa)c_1^p
}
K(R,\tau,T)^{p/2}
\right]V(\bm y_0).
\end{aligned}
\]\end{small}
Combining this estimate with \eqref{eq_4_19} proves
\eqref{eq_4_20b}.
\end{proof}

\subsection{Iterative extension and almost sure exponential stability on the non-exit event}
\label{subsec_4_3}
To extend the preceding finite-time estimates to the infinite time
horizon, we apply them successively on intervals of length \(T\) and
introduce the associated block events for the sampled-data system. Fix
$
R>\|\bm y_0\|,
T>0,
\tau>0,
\frac{T}{\tau}\in\mathbb N.
$

\subsubsection{Iterative Contraction via Conditional Expectations and Exit Probability Estimates}
For each integer \(n\geq0\), let \(\bm y_t\) be the solution of
\eqref{eq_2_2}. On the interval \([nT,(n+1)T]\), introduce an auxiliary
process \(\bm x_t^{(n)}\) as the solution of \eqref{eq_2_1}, driven by
the same Brownian motion and initialized at
$\bm x_{nT}^{(n)}=\bm y_{nT}$. Define the stopping times
\begin{small}\begin{equation}\label{eq_4_2135}
\begin{aligned}
\xi_1^{(n)}
&\triangleq
\inf\left\{
t\geq nT:
\|\bm x_t^{(n)}\|\geq R
\right\},
\xi_2^{(n)}
\triangleq
\inf\left\{
t\geq nT:
\|\bm y_t\|\geq R
\right\},
\end{aligned}
\end{equation}\end{small}where \(\inf\varnothing\triangleq+\infty\).
For each \(n\geq0\), define the block event
$
G_n
\triangleq
\Big\{
(n+1)T$ $<
\xi_1^{(n)}\wedge\xi_2^{(n)}
\Big\}.
$
We then define recursively for $n\geq 0$,
$
A_0\triangleq\Omega,
A_{n+1}\triangleq A_n\cap G_n.$
Finally, set
$
A_\infty
\triangleq
\bigcap_{n=0}^{\infty}A_n
=
\bigcap_{n=0}^{\infty}G_n.
$

For \(0<p\leq1\), let
\begin{small}\begin{equation}\label{eq_4_28}
c(R,\tau,T)
\triangleq
\mathrm {\mathrm e}^{-\lambda_RT}
+
\frac{2^{1-p}c_{2,R}}{pc_1}
K(R,\tau,T)^{p/2},
\beta(R,\tau,T)
\triangleq
\frac{2}{c_1R^p}
+
\frac{2^{1-p}c_{2,R}}{pc_1^2R^p}
K(R,\tau,T)^{p/2}.
\end{equation}\end{small}
For \(1<p<2\), let
\begin{small}\begin{equation}\label{eq_4_28b}\begin{aligned}
c(R,\tau,T,\varkappa)
&\triangleq
\frac{1+\varkappa}{1-\varkappa}
\mathrm {\mathrm e}^{-\lambda_RT}
+
\frac{
2(p-1)^{p-1}c_{2,R}^{p}
}{
p^p\varkappa^{p-1}(1-\varkappa)c_1^p
}
K(R,\tau,T)^{p/2},\\
\beta(R,\tau,T,\varkappa)
&\triangleq
\frac{2}{(1-\varkappa)c_1R^p}
+
\frac{
2(p-1)^{p-1}c_{2,R}^{p}
}{
p^p\varkappa^{p-1}(1-\varkappa)c_1^{p+1}R^p
}
K(R,\tau,T)^{p/2}. \end{aligned}
\end{equation}\end{small}
We next establish the iterative contraction estimate and the associated
exit probability bound.

\begin{lemma}\label{lem_4_10}
Suppose that Assumptions~\ref{ass_2_1} and \ref{ass_2_2} hold with \(0<p<2\).
For brevity, write \(c\) and \(\beta\) for the corresponding quantities in \eqref{eq_4_28} and \eqref{eq_4_28b}
and assume that \(c<1\). Then, for every integer \(n\geq0\),
\begin{small}\begin{equation}\label{eq_4_37}
\mathbb E\left[
V(\bm y_{nT})\cdot \mathbf 1_{A_n}
\right]
\leq
c^nV(\bm y_0).
\end{equation}\end{small}
Moreover,
\begin{small}\begin{equation}\label{eq_4_38}
\mathbb P(A_\infty^c)
\leq
\frac{\beta}{1-c}V(\bm y_0).
\end{equation}\end{small}
\end{lemma}

\begin{proof}
Shifting Lemma~\ref{lem_4_6} to $[nT,(n+1)T]$ yields that 
\begin{small}
\begin{equation}\label{eq_conditional_contraction}
\mathbb E\left[
V(\bm y_{(n+1)T})\cdot \mathbf 1_{G_n}\mathbf \cdot \mathbf 1_{\|\bm{y}_{nT}\|<R}
\,\middle|\,\mathcal F_{nT}
\right]
\leq
cV(\bm y_{nT})
\cdot \mathbf 1_{\|\bm{y}_{nT}\|<R},
\end{equation}\end{small}
whereas Lemma~\ref{lem_4_7} gives
\begin{small}\begin{equation}\label{eq_conditional_exit}
\mathbb E\left[
\mathbf 1_{G_n^c}\mathbf 1_{\|\bm{y}_{nT}\|<R}\,\middle|\,\mathcal F_{nT}
\right]
\leq
\beta V(\bm y_{nT})\cdot \mathbf 1_{\|\bm{y}_{nT}\|<R}.
\end{equation}\end{small}
By construction, $A_n\in\mathcal F_{nT}$,
$G_n\in\mathcal F_{(n+1)T}$, and
$A_n\subset\{\|\bm y_{nT}\|<R\}$. Moreover,
\begin{small}$
\mathbf 1_{A_{n+1}}=\mathbf 1_{A_n}\mathbf 1_{G_n}, \quad \mathbf 1_{A_n} \mathbf 1_{\|\bm y_{nT}\|<R}=\mathbf 1_{A_n}.
$\end{small}
Multiplying $\mathbf 1_{A_n}$ on both sides of  \eqref{eq_conditional_contraction} and  using the tower property 
 yields
 \begin{small}\[
\begin{aligned}
&\mathbb E\left[
V(\bm y_{(n+1)T})
\cdot\mathbf 1_{A_{n+1}}
\,\middle|\,\mathcal F_{nT}\right]=
\mathbb E\left[
V(\bm y_{(n+1)T})
\cdot\mathbf 1_{A_n}\cdot\mathbf 1_{G_n}\cdot\mathbf 1_{\|\bm y_{nT}\|<R}
\,\middle|\,\mathcal F_{nT}\right]\\
&\leq
c V(\bm y_{nT})\cdot\mathbf 1_{A_n}\cdot\mathbf 1_{\|\bm y_{nT}\|<R}=cV(\bm y_{nT})\cdot\mathbf 1_{A_n},
\end{aligned}
\]\end{small} Taking expectations on both sides yields that
\begin{small}$$
\mathbb E\left[V(\bm y_{(n+1)T})
\cdot \mathbf 1_{A_{n+1}}\right]\leq c \mathbb{E}\left[V(\bm y_{nT}) \cdot\mathbf 1_{A_n}\right]\leq\cdots\leq c^{n+1}V(\bm{y}_0).
$$\end{small}
 This proves \eqref{eq_4_37}.  Multiplying $\mathbf 1_{A_n}$ on both sides of  \eqref{eq_conditional_exit}  and taking expectations yields that
$\mathbb E\left[
\mathbf 1_{A_n}\mathbf 1_{G_n^c}
\right]\leq
\beta \mathbb{E}\left[V(\bm y_{nT})\cdot\mathbf 1_{A_n}\right]\leq \beta c^nV(\bm{y}_0)$.
Consequently,
\begin{small}\[
\mathbb P(A_\infty^c)=
\sum_{n=0}^{\infty}
\mathbb P(A_n\cap G_n^c)=
\sum_{n=0}^{\infty}
\mathbb E\left[
\mathbf 1_{A_n}\mathbf 1_{G_n^c}
\right]\leq\sum_{n=0}^\infty \beta c^nV(\bm{y}_0)
=\dfrac{\beta}{1-c}V(\bm{y}_0).\]\end{small} This proves
\eqref{eq_4_38}.
\end{proof}

\subsubsection{Almost sure exponential stability on the non-exit event}
We now establish the exponential convergence of the sampled-data system on the set \(A_\infty\). 

\begin{lemma}\label{lem_4_13}
Under the same assumptions as in Lemma \ref{lem_4_10}, we have
\begin{small}\[
\limsup_{t\to+\infty}
\frac1t\log\|\bm y_t\|
\le
\frac{\log c}{pT}<0
\quad
\text{{\rm a.s.} on }A_\infty.
\]\end{small}
\end{lemma}

\begin{proof}
By Assumption \ref{ass_2_2}(ii)(iv) and \eqref{eq_4_37}, 
$
\mathbb E\left[
\|\bm y_{nT}\|^p\cdot\mathbf{1}_{A_n}
\right]
\le
\frac{c_{2,R}}{pc_1}
\|\bm y_0\|^p c^n.
$ For \(a\in(0,1]\), define \(\zeta_a\triangleq\inf\{t\geq0:\|\bm y_t\|\geq a\}\). Following the method of Lemma \ref{lem_4_3}, we obtain
$
\mathbb{E}\|\bm{y}_{T\wedge\zeta_a}\|^2\leq C\|\bm{y}_0\|^2,
$
where \(C\) is a constant independent of \(a\in(0,1]\). H\"older's inequality yields that
$
\mathbb{E}\|\bm{y}_{T\wedge\zeta_a}\|^p\leq C\|\bm{y}_0\|^p.
$
Thus,
$
\mathbb{P}\left(\sup_{t\in[0,T]}\|\bm{y}_t\|\geq a\right)
=
\mathbb{P}(\zeta_a\leq T)
\leq
\frac{C\|\bm{y}_0\|^p}{a^p}.
$
We then follow the argument in the proof of Lemma~\ref{lem_3_5} and
apply the Borel--Cantelli lemma. The remaining details are omitted.
\end{proof}

\subsection{Parameter selection and main result}
\label{subsec_4_4}
By suitably choosing the parameters, we establish the main result of
this section. 
\begin{theorem}\label{thm_4_4}
Suppose that Assumptions~\ref{ass_2_1} and \ref{ass_2_2}
hold with \(0<p<2\).  For every
\(\varepsilon\in(0,1)\) and every
\(\bm y_0\in\mathbb R^n\backslash\{\bm{0}\}\), there exists $\tau^*=\tau^*(V(\bm{y}_0),\varepsilon)>0$ such that for $\tau\in(0,\tau^*)$,
\begin{small}\begin{equation}\label{eq_4_40}
\mathbb P\left(
\{\rho_\infty=+\infty\}
\cap
\left\{
\limsup_{t\to+\infty}
\frac1t\log\|\bm y_t\|<0
\right\}
\right)
\ge
1-\varepsilon.
\end{equation}\end{small}
\end{theorem}

\begin{proof}
We only discuss the case \(p\in(0,1]\). The case \(p\in(1,2)\) is similar. Fix
\(\varepsilon\in(0,1)\) and
\(\bm y_0\in\mathbb R^n\setminus\{\bm 0\}\). We first choose \(R\)
sufficiently large such that
\begin{small}\begin{equation}
R>\max\left\{\|\bm{y}_0\|, \left(\dfrac{6V(\bm{y}_0)}{c_1\varepsilon}\right)^{1/p}\right\}.
\end{equation}\end{small} Then we choose \(T^*\) such that \({\rm e}^{-\lambda_R T^*}=\frac{1}{4}\). Define
$
T(\tau)\triangleq\inf\{T\geq T^* \mid {T}/{\tau}\in \mathbb{N}\}.
$
Thus \(T^*\leq T(\tau) \leq T^*+\tau\). As \(\tau\to 0+\),
$
c(R,\tau,T(\tau))\leq \frac{1}{4}+O(\tau^{p/2}), 
\beta(R,\tau,T(\tau))\leq\frac{2+O(\tau^{p/2})}{c_1R^p}.
$
Therefore, we choose \(\tau^*>0\) such that, for every
\(\tau\in(0,\tau^*)\),
\begin{small}\begin{equation}
c(R,\tau,T(\tau))<\dfrac{1}{2}, \qquad
\beta(R,\tau,T(\tau))<\dfrac{3}{c_1R^p}.
\end{equation}\end{small}
By Lemma \ref{lem_4_10},
\begin{small}\[
\mathbb P(A_\infty^c)
\leq
\frac{\beta(R,\tau,T(\tau))}{1-c(R,\tau,T(\tau))}V(\bm y_0)
\leq \frac{3/c_1R^p}{1-1/2}V(\bm y_0)
<\varepsilon.
\]\end{small}
Combining this with Lemma \ref{lem_4_13} yields \eqref{eq_4_40}.
\end{proof}

\section{Numerical simulations}\label{sec_5}

In this section, we present two scalar numerical examples. The purpose is threefold: to illustrate the high-probability stabilization result in Theorem~\ref{thm_3_1}, to exhibit exceptional sample paths that rule out almost-sure stability, and to compare the sufficient sampling bounds from our proof with Monte Carlo observations. The two examples share the same continuous-time closed-loop system but differ in their sampled-data implementations, representing two distinct failure mechanisms under sampling: in the first example, the sampled-data solution may explode in finite time, whereas the second example exhibits oscillatory divergent behavior without finite-time explosion.

\subsection{Common Lyapunov setting and numerical protocol}
In both examples, the feedback is chosen such that the
continuous-time closed loop is
\begin{equation}\label{eq_num_common}
    \mathrm d x_t =
    \left(-x_t-x_t^3\right)\mathrm dt+x_t^2\,\mathrm dB_t .
\end{equation}
Let
\(
    V(x)=x^2.
\)
Then
\(
    \mathcal LV(x) =-2x^2-x^4
    \leq -2V(x).
\)
Hence Assumption~\ref{ass_2_2} is satisfied with
\( p=2,
    c_1=1,
    c_{2,R}=c_{3,R}=2,
    \lambda_R=2 ,
\)
for every \(R>0\).  

\textcolor{black}{For the finite-horizon Monte Carlo experiments, we take
\(x_0\in\{0.25,0.5,1\}\),
\(\tau\in\{0.02,0.04,\ldots,0.20\}\), and \(T=8\),
and numerically simulate
\(N=5000\) independent trajectories for each pair \((x_0,\tau)\).}
A trajectory is classified as successful if it satisfies the following three conditions: (i) it remains numerically finite on \([0,T]\); (ii) \(|y_T|<0.05\); and (iii) the least-squares slope of \(\log(|y_t|\vee 10^{-12})\) over the final \(30\%\) of the observation interval is negative. The empirical stability success probability is then defined as the fraction of the \(N\) simulated paths satisfying these conditions.

\subsection{Two illustrative examples}
\subsubsection{Example 1: non-explosive continuous case but explosive under sampling}
We first return to the explosive, physically motivated Example~\ref{ex_1_1} and take
\(
    a=k=1.
\)
Thus the continuous closed loop is exactly
\eqref{eq_num_common}.  Under sampled feedback,
\(\alpha(y_{\delta_t})\), the system becomes
\begin{equation}\label{eq_num_ex1}
    \mathrm dy_t
    =
    \left(
      y_t^2+y_t^3
      -y_{\delta_t}
      -y_{\delta_t}^2
      -2y_{\delta_t}^3
    \right)\mathrm dt
    +y_t^2\,\mathrm dB_t, \quad \delta_t
    =
    \left\lfloor\frac{t}{\tau}\right\rfloor\tau .
\end{equation} Figure~\ref{fig_num_ex1}(a) shows representative successful paths with  exponential decay after a short transient.  Panel~(b) illustrates the explosion mechanism through the reciprocal process \(z_t=1/y_t\), for which explosion corresponds to the hitting event \(z_t=0\), which precisely verifies the conclusion of Lemma \ref{ex_5_1}.

Panel~(c) shows the theoretical sufficient threshold \(\tau_1^*(x_0,\varepsilon)\), which decreases with \(x_0\) for fixed \(1-\varepsilon\), as computed from \eqref{eq_3_21} and \eqref{1818} in Theorem \ref{thm_3_1}. Panel~(d) gives finite-horizon Monte Carlo probabilities, which remain near one over the tested range but decrease for larger \(x_0\) or \(\tau\). The gap between (c) and (d) reveals the conservatism of the theoretical bounds.

\begin{figure}[htbp]
    \centering
    \includegraphics[width=\textwidth]{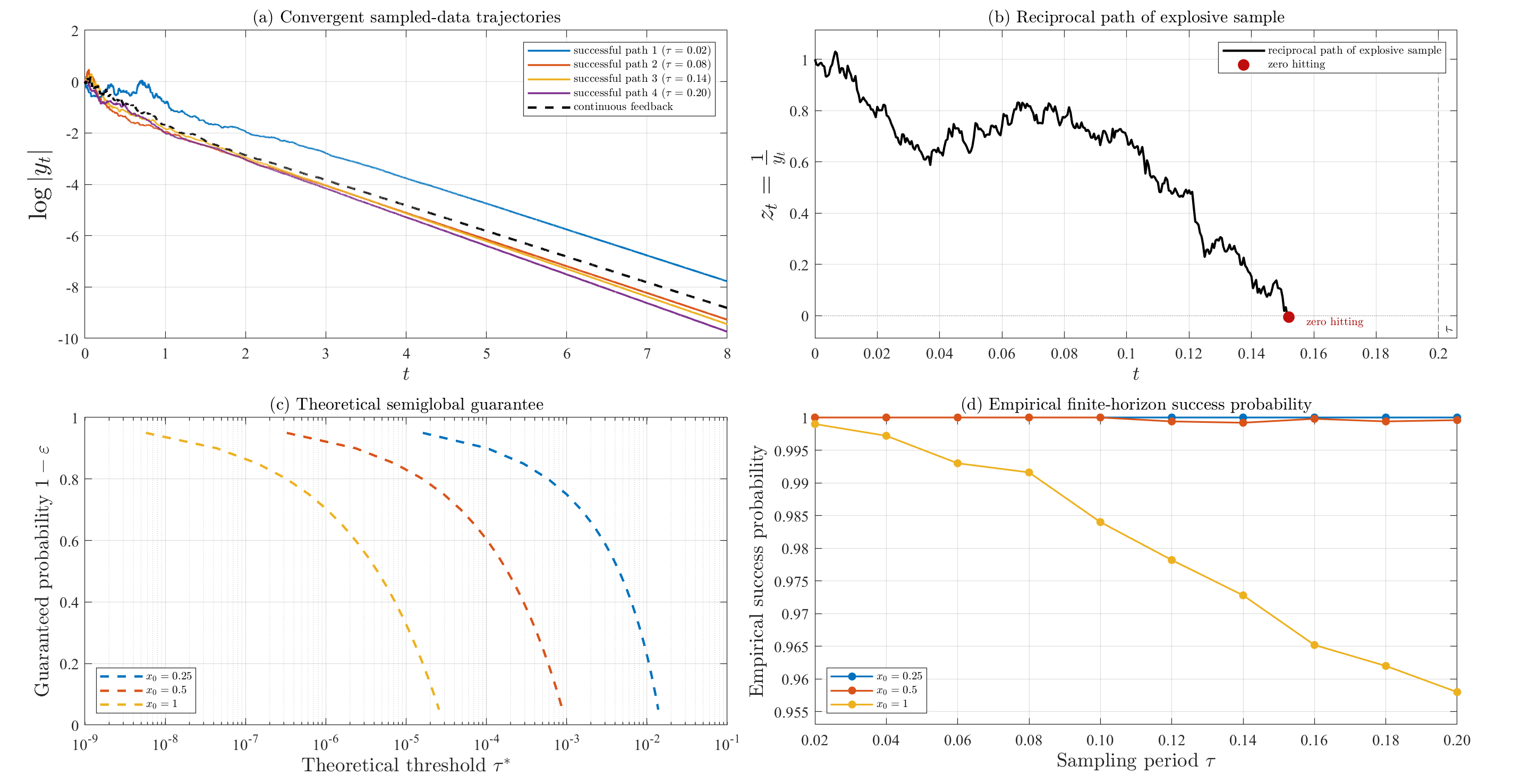}
    \caption{
Numerical illustration of Eq. \eqref{eq_num_ex1}.
(a) Representative convergent sampled-data paths and the continuous-time counterpart with $y_0=1$.
(b) Reciprocal path \(z_t=1/y_t\), where \(z_t=0\) corresponds to explosion with $z_0=1$.
(c) Theoretical sufficient thresholds \(\tau_1^*(x_0,\varepsilon)\) computed from Theorem \ref{thm_3_1} for different initial values.
(d) Empirical stability success probabilities estimated from \(5000\) independent Monte Carlo simulations for each \((x_0,\tau)\).}
    \label{fig_num_ex1}
\end{figure}

\subsubsection{Example 2: non-explosive, but with some oscillatory divergent paths} We return to the controlled CEV-type model (Example~\ref{ex_2_2}) with feedback \(\alpha(x)=-2x-x^3\), whose continuous-time closed loop is again \eqref{eq_num_common}. In contrast,
the sampled-data implementation is
\begin{equation}\label{eq_num_ex2}
    \mathrm dy_t
    =
    \left(
      y_t-2y_{\delta_t}-y_{\delta_t}^3
    \right)\mathrm dt
    +
    y_t^2\,\mathrm dB_t .
\end{equation}
Eq. \eqref{eq_num_ex2} does not explode in finite time, but numerical evidence shows that some sample paths are oscillatory and unbounded above, with \(\limsup_{t\to\infty} |y_t| = \infty\) and \(\liminf_{t\to\infty} |y_t| = 0\), which precisely verifies the conclusion of Lemma \ref{ex_5_2}.

Figure~\ref{fig_num_ex2}(a) shows representative convergent paths with exponential decay after a short transient. Panel~(b) displays a divergent sample path that remains finite over the simulated time horizon but oscillates between near-zero values and arbitrarily large amplitudes, and hence does not converge to the origin. Panel~(c) plots the theoretical sufficient threshold \(\tau_2^*(x_0,\varepsilon)\), computed from  Theorem \ref{thm_3_1}. Panel~(d) shows the finite-horizon empirical success probabilities from \(5000\) Monte Carlo simulations, which remain close to one over the tested range. As in Example~1, the gap between panels~(c) and~(d) illustrates the conservativeness of the sufficient theoretical bounds.

\begin{figure}[htbp]
    \centering
    \includegraphics[width=\textwidth]{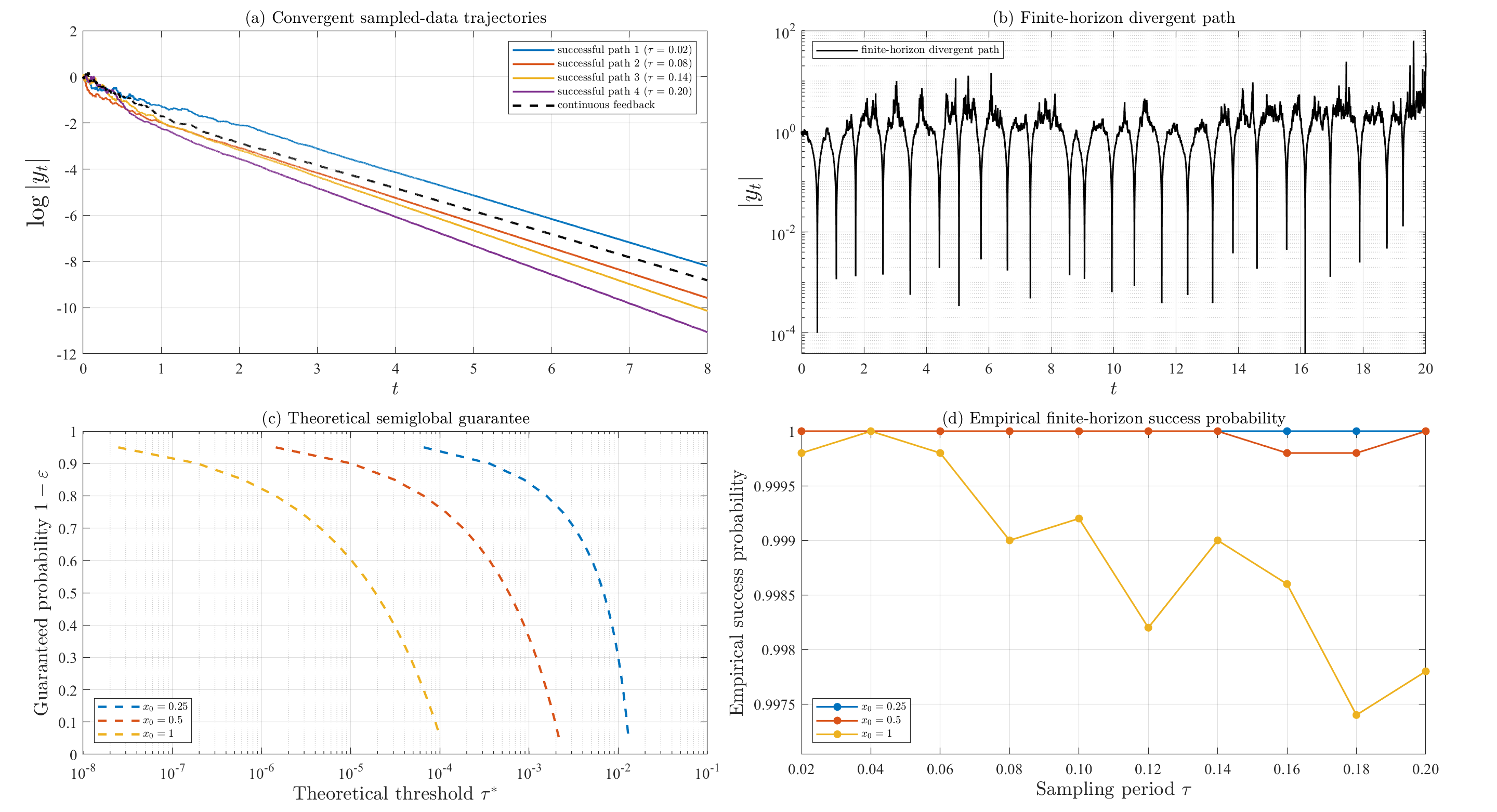}
    \caption{
    Numerical illustration of Eq. \eqref{eq_num_ex2}.
(a) Representative convergent sampled-data paths and a continuous-feedback path with $x_0=1$.
(b) A representative  oscillatory divergent path with $x_0=1$ that remains finite over the simulated interval.
(c) Theoretical sufficient threshold \(\tau_2^*(x_0,\varepsilon)\) computed from Theorem \ref{thm_3_1}.
(d) Empirical stability success probabilities from \(5000\) Monte Carlo paths for each \((x_0,\tau)\).}
\label{fig_num_ex2}
\end{figure}

\section{Concluding Remarks}\label{sec_6}

This paper studies sampled-data stabilization of nonlinear stochastic systems under a local exponential-type Lyapunov condition on the continuous-time closed-loop system, with only locally Lipschitz coefficients and without imposing linear growth or Khasminskii-type conditions. Two counterexamples show that almost sure stability is generally not attainable. Instead, for any initial state, we prove exponential stability with arbitrarily high probability under sufficiently fast sampling. The cases \(p\ge 2\) and \(0<p<2\) are treated using a generalized Halanay inequality and a comparison method, respectively. Future work includes deriving sharper sampling bounds and extending the results to broader control frameworks.


\section{Appendix}

\subsection{Appendix A}\label{subsec_5_1}
In this subsection, we give a proof of Lemma \ref{lem_3_1}.

\begin{proof}[Proof of Lemma~\ref{lem_3_1}] 
Define $
F(\mu)\triangleq \mu-a+b{\mathrm e}^{\mu\tau}$.
Since $a>b$, $F(0)=b-a<0$.
Moreover,
$
F'(\mu)=1+b\tau {\mathrm e}^{\mu\tau}>0,
$
and $F(\mu)\to\infty$ as $\mu\to\infty$. Hence there is a unique
$\mu>0$ such that $F(\mu)=0$, which is exactly $\mu=a-b{\mathrm e}^{\mu\tau}$.

Set $
C\triangleq\sup_{-\tau\le s\le0}\Phi(s)$ and define the comparison function
$
W(t)\triangleq C{\mathrm e}^{-\mu t},\ \ t\ge-\tau.
$
Since $W$ is nonincreasing,
$
\sup_{-\tau\le\theta\le0}W(t+\theta)
=W(t-\tau)
={\mathrm e}^{\mu\tau}W(t).
$
Using $\mu=a-b{\mathrm e}^{\mu\tau}$, we obtain
\begin{small}\begin{equation}
W'(t)
=-aW(t)
+b\sup_{-\tau\le\theta\le0}W(t+\theta),
\quad t\ge0.
\label{eq:W-eq}
\end{equation}\end{small}
We now prove that $\Phi(t)\le W(t)$ for all $t\ge0$. Fix $\epsilon>0$ and set
$A_{\epsilon}\triangleq 
\{t\ge0:\ \Phi(t)>W(t)+\epsilon\}.
$
Suppose, for contradiction, that $A_{\epsilon}\neq\varnothing$, and let
$
t_*\triangleq \inf A_{\epsilon}.
$
For every $t\in[-\tau,0]$, the definition of $C$ gives
$
\Phi(t)\le C\le C{\mathrm e}^{-\mu t}=W(t)<W(t)+\epsilon.
$
By the definition of $t_*$, we also have $\Phi(t)\le W(t)+\epsilon$ for
$0\le t<t_*$.  By the continuity of $W$,
$
\liminf_{t\to t_*^-}\Phi(t)
\le W(t_*)+\epsilon.
$
By the second inequality in \eqref{eq:jump},
$
\Phi(t_*)\le\liminf_{t\to t_*^-}\Phi(t),
$
and hence
\begin{small}\begin{equation}\label{eq:contact-upper}
\Phi(t_*)\le W(t_*)+\epsilon.
\end{equation}\end{small}
Thus $t_*\notin A_{\epsilon}$. Since $t_*=\inf A_{\epsilon}$, there exists a
sequence $t_n\in A_{\epsilon}$ such that $t_n\downarrow t_*$. Hence
$
\Phi(t_n)>W(t_n)+\epsilon.
$
It follows that
$
\limsup_{t\to t_*^+}\Phi(t)
\ge \limsup_{n\to\infty}\Phi(t_n)
\ge W(t_*)+\epsilon.
$
By the first inequality in \eqref{eq:jump},
$
\limsup_{t\to t_*^+}\Phi(t)\le\Phi(t_*),
$
so
\begin{small}\begin{equation}
\Phi(t_*)\ge W(t_*)+\epsilon.
\label{eq:contact-lower}
\end{equation}\end{small}
Combining \eqref{eq:contact-upper} and \eqref{eq:contact-lower}, we obtain
\begin{small}\begin{equation}
\Phi(t_*)=W(t_*)+\epsilon.
\label{eq:contact}
\end{equation}\end{small}
By the definition of the upper right Dini derivative and by \eqref{eq:contact},
\begin{small}\begin{equation}
\begin{aligned}
&D^+\Phi(t_*)
\ge \limsup_{n\to\infty}
\frac{\Phi(t_n)-\Phi(t_*)}{t_n-t_*}\ge \limsup_{n\to\infty}
\frac{W(t_n)-W(t_*)}{t_n-t_*}
=W'(t_*)\\&=-aW(t_*)
+b\sup_{-\tau\le\theta\le0}W(t_*+\theta)\ge -a\bigl(\Phi(t_*)-\epsilon\bigr)
+b\left(
\sup_{-\tau\le\theta\le0}\Phi(t_*+\theta)-\epsilon
\right)\\
&=-a\Phi(t_*)
+b\sup_{-\tau\le\theta\le0}\Phi(t_*+\theta)
+(a-b)\epsilon,
\end{aligned}
\label{eq:dini-lower}
\end{equation}\end{small}
which contradicts \eqref{eq:dini-ineq}. Therefore
$A_{\epsilon}=\varnothing$. Letting $\epsilon\to 0$ yields the result.
\end{proof}

\begin{remark}
Condition \eqref{eq:jump} allows discontinuities but rules out upward jumps:
every right-hand cluster value is no larger than $\Phi(t)$, while every
left-hand cluster value is no smaller than $\Phi(t)$. These are precisely
the one-sided regularity properties used in the first-crossing argument.
\end{remark}

\subsection{{\bf Appendix B}}
\label{subsec_5_4}
We prove the claims of the two examples in the Introduction. Given any nonzero initial value and any \(\tau>0\), almost sure stability fails. The proof relies on the support theorem (Theorem~\ref{thm_5_2}; see also \cite[(5.1)--(5.2)]{StroockVaradhan1972Support}), which states that for an SDE in Stratonovich form, its solutions approximate those of the associated deterministic ODE with positive probability (cf. \cite[Theorem~3.3]{LiangAminiMason2019}).
\begin{theorem}[Support theorem]
\label{thm_5_2}
Let \(X_0(t,x)\) be a bounded measurable function, uniformly Lipschitz
continuous in \(x\), and let \(X_k(t,x)\), \(k=1,\ldots,n\), be continuously
differentiable in \(t\) and twice continuously differentiable in \(x\), with
bounded derivatives. Consider the Stratonovich equation
\begin{small}\[
    \mathrm dx_t
    =
    X_0(t,x_t)\,\mathrm dt
    +
    \sum_{k=1}^n X_k(t,x_t)\circ \mathrm dW_t^k,
    \quad x_0=x .
\]\end{small}
Let \(\mathbb P_x\) be the probability law of the solution \(x_t\) starting at
\(x\). Consider, in addition, the associated deterministic control system
\begin{small}\[
    \frac{\mathrm d}{\mathrm dt}x_v(t)
    =
    X_0(t,x_v(t))
    +
    \sum_{k=1}^n X_k(t,x_v(t))v^k(t),
    \quad x_v(0)=x,
\]\end{small}
with \(v^k\in\mathcal V\), where \(\mathcal V\) is the set of all piecewise
constant functions from \(\mathbb R_+\) to \(\mathbb R\).

Let \(\mathcal W_x\) be the set of all continuous paths from \(\mathbb R_+\)
to \(\mathbb R\) starting at \(x\), equipped with the topology of uniform
convergence on compact sets, and let \(\mathcal I_x\) be the smallest closed
subset of \(\mathcal W_x\) such that $\mathbb P_x(x_\cdot\in\mathcal I_x)=1$.
Then
\begin{small}\[
    \mathcal I_x
    =
    \overline{
    \{x_v(\cdot)\in\mathcal W_x:\ v\in\mathcal V^n\}
    }
    \subset \mathcal W_x .
\]\end{small}
\end{theorem}

The preceding theorem immediately implies the following fact. If
\(x_v(\cdot)\) is a controlled trajectory generated by some
\(v\in\mathcal V^n\), then \(x_v(\cdot)\in\mathcal I_x\). Hence, for every \(T>0\) and every \(\eta>0\),
  $  \mathbb P_x\left(
    \sup_{0\le t\le T}|x_t-x_v(t)|<\eta
    \right)>0$.

\begin{lemma}
\label{ex_5_1}
Consider the \(1\)-dimensional stochastic system
\begin{small}\begin{equation}\label{eq_5_1}
    \mathrm dx_t
    =
    \left(x_t^2+x_t^3-C\right)\,\mathrm dt
    +
    x_t^2\,\mathrm dB_t,
    \quad x_0>0.
\end{equation}\end{small}
Here, \(C\) is a constant. Let the stopping time
\begin{small}\(
    \sigma_N\triangleq \inf\{t\ge0:\ |x_t|\ge N\},
\) \end{small}
and let \begin{small}\(\sigma_\infty\triangleq \lim_{N \rightarrow + \infty} \sigma_N\)\end{small} denote the explosion time of the maximal local solution. Then for any \(\tau>0\), 
\begin{small}$
    \mathbb P(\sigma_\infty<\tau)>0.
$\end{small}
This implies that system \eqref{ex_1_1_1} with a sampled-data control will explode in any arbitrarily short time with positive probability. 
\end{lemma}

\begin{proof}
Let \begin{small}\(y_t = \dfrac{1}{x_t}\)\end{small}. Then \(y_t\) satisfies the following SDE on the interval \([0, \sigma_\infty \wedge \theta_0)\):
\begin{small}\begin{equation}\label{eq_5_2}
{\mathrm d}y_t = (-1 + Cy_t^2)\,{\mathrm d}t - {\mathrm d}B_t, \quad y_0 = \dfrac{1}{x_0},
\end{equation}\end{small}
where \(\theta_0 \triangleq \inf\{t \ge 0 : |x_t| = 0\}\). Choose \(f\in C_b^\infty(\mathbb R)\) such that \(f(r)=-1+Cr^2\) for \( r\in[-2,\dfrac{1}{x_0}+1]\). We consider the auxiliary process
\begin{small}\begin{equation}\label{eq_5_25}
    \mathrm dz_t
    =
    f(z_t)\,\mathrm dt-\mathrm dB_t,
    \quad
    z_0=\frac1{x_0}>0.
\end{equation}\end{small}
In the following, we use Theorem~\ref{thm_5_2} to prove that the solution \(z_t\) of Eq. \eqref{eq_5_25} satisfies \(\mathbb{P}(\kappa_0<\tau\wedge\kappa_1)>0\), where \(\kappa_0\triangleq\inf\{t>0: z_t=0\}, \kappa_1\triangleq\inf\{t>0: z_t=\frac{1}{x_0}+\frac{1}{2}\}\). This suffices to show that \(\mathbb{P}(\sigma_\infty<\tau)>0\) since \(y_t=z_t\) on the interval \([0,\kappa_0\wedge\kappa_1)\).

By Theorem~\ref{thm_5_2}, we consider the associated deterministic control system 
\begin{small}\begin{equation}\label{1228}
    \dot r(t)
    =
    f(r(t))-v(t),
    \quad
    r(0)=\dfrac{1}{x_0}.
\end{equation}\end{small}
Choose a constant control function \(v(t)\equiv M\), where
$
    M>
    \sup_{r\in\mathbb R}|f(r)|
    +
    \frac{2(z_0+1)}{\tau}.
$
Then the corresponding controlled trajectory of Eq. \eqref{1228} satisfies
$
    \dot r(t)
    \le
    \sup_{r\in\mathbb R}|f(r)|-M
    <
    -\frac{2(z_0+1)}{\tau}.
$
Set \(t_1\triangleq \inf\{t\ge0:\ r(t)=-1\}\). It is easy to verify that \(t_1<\frac{\tau}{2}\), and \(r(t)\) is monotonically decreasing on \([0,t_1]\) while \(r(t_1)=-1\). By Theorem~\ref{thm_5_2}, we have
$
    \mathbb P\left(
    \sup_{0\le t\le t_1}|z_t-r(t)|<\frac12
    \right)>0.
$
Since \(r(t)\in[-1,z_0]\) for \(t\in[0,t_1]\) and \(r(t_1)=-1\). Then on the event \(A \triangleq \left\{ \sup_{0\le t\le t_1}|z_t-r(t)|<\frac12 \right\}\), we have \(z_t\in(-\frac{3}{2},z_0+\frac{1}{2})\) and \(z_{t_1}<-\frac{1}{2}\). Thus, we know \(A \subseteq \{\kappa_0<\tau\wedge\kappa_1\}\) and consequently \(\mathbb{P}(\{\kappa_0<\tau\wedge\kappa_1\})\geq \mathbb P(A)>0\).
\end{proof}

\begin{lemma}
\label{ex_5_2}
Consider the stochastic hybrid system $(\kappa>1)$
\begin{small}\begin{equation}\label{eq_5_26}
    \mathrm dx_t
    =
    \left[ x_t-\kappa x_{t_n}-x_{t_n}^3
    \right]\mathrm dt
    +
 x_t^2\,\mathrm dB_t 
    , \quad t \in [t_n,t_{n+1}],
\end{equation}\end{small} with initial value \(x_0\neq 0\) and sampling instants \(t_n = n\tau\). Then, for every \(\tau>0\), the solution exists for all \(t\ge 0\), and moreover
$$
\textcolor{black}{\mathbb{P}\left(
\lim_{n\to+\infty}|x_{t_n}|=+\infty, \
\liminf_{t\to+\infty}|x_t|=0
\right)>0.}
$$
\end{lemma}

\begin{proof}
On each sampling interval, Eq.~\eqref{eq_5_26} can be written as ${\rm d}x_t=(x_t+C){\rm d}t+x_t^2{\rm d}B_t$. Thus, for $\Phi(x)=\log(1+x^2)$,  $\mathcal{L}\Phi(x)\to-\infty$ as $|x|\to+\infty$. Therefore, \(\mathcal{L}\Phi(x)\) is bounded above, and hence the solution exists on each interval. By induction, the solution exists globally. The proof of the remaining part is divided into three steps.

\textbf{Step 1:} On the first sampling interval $[0,\tau]$, we prove \begin{small}$\mathbb{P}(|x_\tau|>R)>0$\end{small} for large $R>0$.  To this end, we rewrite \eqref{eq_5_26} in the Stratonovich form on $[0,\tau]$:
\begin{small}$$\label{1504}
{\mathrm d}x_t=(x_t-{x_t^3}-C){\mathrm d}t+x_t^2\circ \mathrm dB_t,$$\end{small}
where $C\triangleq \kappa x_{0}+x_{0}^3\neq 0$. Consider the associated deterministic control system
\begin{small}\[
    \dot r(t)=r(t)-r(t)^3-C+r(t)^2v(t),
   \quad r(0)=x_0.
\]\end{small} By choosing a constant control $v(t)\equiv M$ (or $-M$)  sufficiently large, we have $|r(\tau)|>R$. Following a similar argument to the proof of Lemma \ref{ex_5_1}, the Support Theorem \ref{thm_5_2} guarantees that $\mathbb{P}(|x_\tau|>R)>0$.

\textbf{Step 2:}
We prove that there exist \(R>0\) and \(\theta,q\in(0,1)\) such that on $[0,\tau]$,
\begin{small}\begin{equation}\label{0949}
|x_0|\geq R
\quad\Longrightarrow\quad
\mathbb E(1+|x_\tau|)^{-\theta}
\leq q(1+|x_0|)^{-\theta}.
\end{equation}\end{small}
Set
\begin{small}$r=|x_0|, x_t=x_0Z_s, s=r^2t,\widetilde B_s=x_0B_{s/r^2}.$\end{small}
Then \eqref{eq_5_26}
on \([0,\tau]\) becomes
\begin{small}\begin{equation}\label{1056}
\mathrm dZ_s
=f_r(Z_s)\,\mathrm ds
+Z_s^2\,\mathrm d\widetilde B_s,\end{equation}\end{small}
with initial value $Z_0=1$, 
where
$
f_r(z)
\triangleq
-1+\frac{z-\kappa}{r^2}$, $\widetilde B_s$ is a standard Brownian motion. The interval \((-\infty,0)\) is invariant under \eqref{1056}.
Moreover, \eqref{1056} admits an invariant probability measure
\(\mu_r\) on \((-\infty,0)\), whose density \(\pi_r\) satisfies the
stationary Fokker--Planck equation
$
-\frac{\mathrm d}{\mathrm dz}\bigl(f_r(z)\pi_r(z)\bigr)
+\frac12\frac{\mathrm d^2}{\mathrm dz^2}\bigl(z^4\pi_r(z)\bigr)
=0$ for $z<0$. Its normalized zero-flux solution is
\begin{small}$
\pi_r(z)
\triangleq
C_r|z|^{-4}
\exp\left\{
-\frac{1}{r^2z^2}
+\frac{2}{3z^3}
\left(1+\frac{\kappa}{r^2}\right)
\right\}\mathbf{1}_{z<0}
$\end{small}
where \(C_r>0\) is the normalizing constant. In what follows, we show that the solution of \eqref{1056} rapidly
enters \((-\infty,0)\) and then use synchronous coupling to establish
its rapid convergence to the invariant distribution. This will yield
\eqref{0949}. Indeed, \eqref{0949} is equivalent to
\begin{small}\begin{equation}\label{1116}
\left(1+\frac{1}{r}\right)^\theta
\mathbb E\left[
\left(
|Z_{r^2\tau}|+\frac{1}{r}
\right)^{-\theta}
\right]
\leq q.
\end{equation}\end{small}
Define the stopping times
$\zeta\triangleq\inf\{s\geq0:Z_s=-1\}$. Then
\begin{small}\[
\begin{aligned}
&\mathbb E\left[
\left(|Z_{r^2\tau}|+\frac1r\right)^{-\theta}
\right]
=
\mathbb E\left[
\left(|Z_{r^2\tau}|+\frac1r\right)^{-\theta};
\zeta>\frac{r^2\tau}{2}
\right]
+
\mathbb E\left[
\left(|Z_{r^2\tau}|+\frac1r\right)^{-\theta};
\zeta\leq\frac{r^2\tau}{2}
\right]\\
\leq{}&
\underbrace{
r^\theta\mathbb P\left(\zeta>\frac{r^2\tau}{2}\right)
}_{I_1}
+
\underbrace{
\mathbb E\left[
\mathbf1_{\{\zeta\leq r^2\tau/2\}}
\mathbb E\left[
\left.
\left(|Z_{r^2\tau}|+\frac1r\right)^{-\theta}
\,\right|\,\mathscr F_\zeta
\right]
\right]
}_{I_2}.
\end{aligned}
\]\end{small}
As to estimate $I_1$, denote by $H(x)=2-\frac{1}{(x+2)^2}$, then $1\leq H(x)<2$ for $x\geq -1$. We compute that $
\mathcal L_rH(x)=
-\frac{2(x+2)+3x^4}{(x+2)^4}
+\frac{2(x-\kappa)}{r^2(x+2)^3}\leq -A+B/r^2$, where $A\triangleq\inf_{x\geq-1}\frac{2(x+2)+3x^4}{(x+2)^4}>0$, $B\triangleq\sup_{x\geq-1}\frac{2(x-\kappa)}{(x+2)^3}<+\infty$. When $r\geq\sqrt{\frac{2B}{A}}$, we have $\mathcal{L}_rH(x)\leq-\frac{A}{2}\leq-\lambda_1 H(x)$, where $\lambda_1\triangleq\frac{A}{4}$. Consequently, we have
\begin{small}$
2>H(Z_0)\geq\mathbb{E}[{\rm e}^{\lambda_1(t\wedge\zeta)}H(Z_{t\wedge\zeta})]\geq {\rm e}^{\lambda_1 t}\mathbb{P}(\zeta>t).
$\end{small} This yields that $I_1\leq 2r^\theta{\rm e}^{-\lambda_1r^2\tau/2}$.

To estimate \(I_2\), consider \eqref{1056} with \(Z_0=-1\) and set
\(Y_s=-1/Z_s\). It\^o's formula gives
\begin{small}$
\mathrm dY_s
=
b_r(Y_s)\,\mathrm ds+\mathrm d\widetilde B_s,
 Y_0=1,$\end{small}
where
\begin{small}$
b_r(y)
\triangleq
\frac1y
-\left(1+\frac{\kappa}{r^2}\right)y^2
-\frac{y}{r^2}$\end{small}. A direct calculation gives
\begin{small}$b_r'(y)=-\frac1{y^2}
-2\left(1+\frac{\kappa}{r^2}\right)y
-\frac1{r^2}
\leq
-\frac1{y^2}-2y
\leq-3$\end{small}. Let \(\overline Z_0\) be independent of \(\widetilde B\) and have
density \(\pi_r\), and set
\(\overline Y_0=-1/\overline Z_0\). Consider the auxiliary process
\begin{small}$\mathrm d\overline Y_s=
b_r(\overline Y_s)\,\mathrm ds+\mathrm d\widetilde B_s,
$\end{small}
driven by the same Brownian motion as \(Y_s\). Since \(b_r'(y)\leq-3\),
\begin{small}$
\mathbb E|Y_s-\overline Y_s|^2
\leq
\mathrm e^{-6s}
\mathbb E|Y_0-\overline Y_0|^2
=
\mathrm e^{-6s}
\mathbb E|1-\overline Y_0|^2.
$\end{small} Using \(|Z_s|=1/Y_s\), \(\theta\in(0,1)\), $\overline{Y}_s$ has the same distribution as $\overline{Y}_0$, we have
\begin{small}\begin{align*}
&\mathbb E\left[
\left(|Z_s|+\frac1r\right)^{-\theta}
\right]
=
\mathbb E\left[
\left(\frac1{Y_s}+\frac1r\right)^{-\theta}
\right]
\leq
\mathbb E\left|
\left(\frac{rY_s}{r+Y_s}\right)^\theta
-
\left(\frac{r\overline Y_s}{r+\overline Y_s}\right)^\theta
\right|
+
\mathbb E\left(
\frac{r\overline Y_s}{r+\overline Y_s}
\right)^\theta\\
&\leq
\mathbb E\left|
\frac{rY_s}{r+Y_s}
-
\frac{r\overline Y_s}{r+\overline Y_s}
\right|^\theta
+
\mathbb E\left(
\frac{r\overline Y_s}{r+\overline Y_s}
\right)^\theta
\leq
\mathbb E|Y_s-\overline Y_s|^\theta
+
\mathbb E\left(
\frac{r\overline Y_s}{r+\overline Y_s}
\right)^\theta\\
&\leq
\mathrm e^{-3\theta s}
\left(\mathbb E|1-\overline Y_0|^2\right)^{\theta/2}
+
\mathbb E\left(
\frac{r\overline Y_0}{r+\overline Y_0}
\right)^\theta
\leq C_1{\rm e}^{-3\theta s}+\mathbb E\left(
\frac{r\overline Y_0}{r+\overline Y_0}
\right)^\theta\\
&=C_1{\rm e}^{-3\theta s}+\mathbb E\left(
|\overline Z_0|+\frac1r
\right)^{-\theta}= C_1{\rm e}^{-3\theta s}+\int_{-\infty}^0
\left(|z|+\frac{1}{r}\right)^{-\theta}\pi_r(z)\mathrm dz.
\end{align*}\end{small}
Choosing $\theta=1/2$, letting $r\to+\infty$, the second term is approximated as
\begin{small}\[
\begin{aligned}
\lim_{r\to\infty}
\mathbb E\left(
\frac{r\overline Y_0}{r+\overline Y_0}
\right)^\theta
&=
2\int_{-\infty}^0
|z|^{-\theta-4}
\exp\left(\frac{2}{3z^3}\right)\mathrm dz=
\left(\frac32\right)^{\theta/3}
\Gamma\left(1+\frac{\theta}{3}\right)\approx0.9926<1.
\end{aligned}
\]\end{small} Therefore, there exist \(q_0\in(0.9926,1)\) and \(R_0\) such that for $r\geq R_0$,
\begin{small}$
\mathbb E\left(
\frac{r\overline Y_0}{r+\overline Y_0}
\right)^{1/2}
\leq q_0.
$\end{small}
By strong Markov property,
$
I_2
\leq
C_1\mathrm e^{-3\theta r^2\tau/2}+q_0$. Combining this with the estimate
of \(I_1\), we obtain
$
I_1+I_2
\leq
2r^\theta\mathrm e^{-\lambda_1r^2\tau/2}
+C_1\mathrm e^{-3\theta r^2\tau/2}+q_0.
$
Choosing \(q\in(q_0,1)\) and then \(R\geq R_0\) sufficiently large
ensures that \eqref{1116} holds for every \(r\geq R\).

\textbf{Step 3:} Denote by $X_n\triangleq x_{n\tau}$, $\mathscr{G}_n\triangleq \mathscr{F}_{n\tau}$. Shifting \eqref{0949} to $[n\tau,(n+1)\tau]$ yields
\begin{small}\begin{equation}\label{0948}
\mathbb{E}[W(X_{n+1})\cdot\mathbf{1}_{|X_n|\geq R}\mid\mathscr{G}_n]\leq qW(X_n)\cdot\mathbf{1}_{|X_n|\geq R},
\end{equation}\end{small} where $W(x)\triangleq(1+|x|)^{-\theta}$.  We first suppose $|X_1|\geq 2R$ as a deterministic value, denote by $\eta\triangleq\inf\{n\in\mathbb{N}\mid |X_n|<R\}$. By $\{\eta\geq n+1\}\subseteq \{|X_n|\geq R\}$,
\begin{small}$$
\mathbb{E}[W(X_{(n+1)\wedge\eta})-W(X_{n\wedge\eta})\mid\mathscr{G}_n]=\mathbb{E}[W(X_{n+1})-W(X_{n})\mid\mathscr{G}_n]\cdot\mathbf{1}_{\eta\geq n+1}\leq 0.
$$\end{small}
Taking expectation yields that \begin{small}$$\mathbb{E}W(X_{(n+1)\wedge\eta})\leq\mathbb{E}W(X_{n\wedge\eta})\leq\cdots \leq W(X_1)\leq W(2R).$$\end{small} Consequently, we obtain \begin{small}$W(2R)\geq\mathbb{E}W(X_{(n+1)\wedge\eta})\geq \mathbb{E}[W(X_\eta);\eta\leq n+1]\geq W(R)\mathbb{P}(\eta\leq n+1).$\end{small} Letting $n\to+\infty$ yields that \begin{small}$\mathbb{P}(\eta<+\infty)\leq W(2R)/W(R)$\end{small}. Thus we have \begin{small}$\mathbb{P}(\eta=+\infty)\geq 1-W(2R)/W(R)>0.$\end{small} Multiplying $\mathbf{1}_{\eta\geq n+1}$ on both sides of \eqref{0948} and taking expectations yields that
\begin{small}$$
\mathbb{E}[W(X_{n+1})\mathbf{1}_{\eta\geq n+2}]\leq \mathbb{E}[W(X_{n+1})\mathbf{1}_{\eta\geq n+1}]\leq q \mathbb{E}[W(X_{n})\mathbf{1}_{\eta\geq n+1}]\leq\cdots\leq q^nW(X_1).
$$\end{small} 
Thus, 
\begin{small}$$\sum_{n=1}^{+\infty}\mathbb{P}(W(X_{n+1})\geq q^{(1-\epsilon)n};\eta\geq n+2)\leq\sum_{n=1}^{+\infty}q^{\epsilon n}W(X_1)<+\infty.$$\end{small}
 By the Borel Cantelli Lemma, \begin{small}$\liminf_{n\to+\infty}\frac{\log |X_n|}{n}\geq -\frac{1-\epsilon}{\theta}\log q>0 \quad\text{a.s. on }\{\eta=\infty\}.$\end{small}
 
 \textcolor{black}{On the other hand, the preceding estimate gives
\begin{small}$
\mathbb P\left(
\inf_{t\in[0,\tau]}|x_t|>0
\right)
\leq
\mathbb P(\zeta>r^2\tau)
\leq
2\mathrm e^{-\lambda_1|x_0|^2\tau}.
$\end{small}
Shifting this estimate to each sampling interval
\([t_n,t_{n+1}]\) and using the preceding exponential-growth result,
we obtain
\begin{small}$$
\sum_{n=1}^{\infty}
\mathbb P\left(
\left.
\inf_{t\in[t_n,t_{n+1}]}|x_t|>0,\,
\eta\geq n+1
\,\right|\,\mathscr{G}_n
\right)\leq
\sum_{n=1}^{\infty}
2\mathrm e^{-\lambda_1|X_n|^2\tau}
\mathbf1_{\{\eta\geq n+1\}}
<\infty
\quad\text{a.s.}.
$$\end{small}
Hence, by the conditional Borel--Cantelli lemma,
\begin{small}$
\liminf_{t\to\infty}|x_t|=0
\quad\text{a.s. on }\{\eta=\infty\}.
$\end{small}}

Finally, we return to the proof of Lemma \ref{ex_5_2}. By \textbf{Step 1}, \begin{small}$\mathbb{P}(\mathcal{A})>0(\mathcal{A}\triangleq \{|X_1|>2R\}).$\end{small} Thus 
 \begin{small}$$
 \mathbb{P}(\limsup_{n\to+\infty}|x_{t_n}|=+\infty,\liminf_{t\to\infty}|x_t|=0)\geq \mathbb{P}(\eta=+\infty;\mathcal{A})=\mathbb{P}(\eta=+\infty\mid\mathcal{A})\mathbb{P}(\mathcal{A})>0.$$\end{small}
 This completes the proof.
\end{proof}

\bibliographystyle{siamplain}
\bibliography{references.bib}

@ARTICLE{LinWei_TAC2020,
  author={Lin, Wei},
  journal={IEEE Transactions on Automatic Control}, 
  title={When Is a Nonlinear System Semiglobally Asymptotically Stabilizable by Digital Feedback?}, 
  year={2020},
  volume={65},
  number={11},
  pages={4584-4599},}

@book{maobook,
  title={Stochastic differential equations and applications},
  author={Mao, Xuerong},
  year={2007},
  publisher={Elsevier}
}

@article{mao2015almost,
  title={Almost sure exponential stabilization by discrete-time stochastic feedback control},
  author={Mao, Xuerong},
  journal={IEEE Transactions on Automatic Control},
  volume={61},
  number={6},
  pages={1619--1624},
  year={2015},
  publisher={IEEE}
}

@article{LiangAminiMason2019,
author = {Liang, Weichao and Amini, Nina H. and Mason, Paolo},
title = {On Exponential Stabilization of \$N\$-Level Quantum Angular Momentum Systems},
journal = {SIAM Journal on Control and Optimization},
volume = {57},
number = {6},
pages = {3939-3960},
year = {2019}
}

@inproceedings{StroockVaradhan1972Support,
  author    = {Stroock, Daniel W. and Varadhan, S. R. S.},
  title     = {On the Support of Diffusion Processes with Applications to the Strong Maximum Principle},
  booktitle = {Proceedings of the Sixth Berkeley Symposium on Mathematical Statistics and Probability},
  volume    = {3},
  pages     = {333--359},
  address   = {Berkeley, California},
  year      = {1972},
  note      = {University of California, Berkeley, California, 1970/1971}
}

@article{Mao2013,
title = {Stabilization of continuous-time hybrid stochastic differential equations by discrete-time feedback control},
journal = {Automatica},
volume = {49},
number = {12},
pages = {3677-3681},
year = {2013},
issn = {0005-1098},
author = {Xuerong Mao}
}

@article{MaoLiuHuLuoLu2014,
title = {Stabilization of hybrid stochastic differential equations by feedback control based on discrete-time state observations},
journal = {Systems \& Control Letters},
volume = {73},
pages = {88-95},
year = {2014},
issn = {0167-6911},
author = {Xuerong Mao and Wei Liu and Liangjian Hu and Qi Luo and Jianqiu Lu}
}

@article{YouLiuLuMaoQiu2015,
author = {You, Surong and Liu, Wei and Lu, Jianqiu and Mao, Xuerong and Qiu, Qinwei},
title = {Stabilization of Hybrid Systems by Feedback Control Based on Discrete-Time State Observations},
journal = {SIAM Journal on Control and Optimization},
volume = {53},
number = {2},
pages = {905-925},
year = {2015},
}

@article{QiuLiuHuMaoYou2016,
title = {Stabilization of stochastic differential equations with Markovian switching by feedback control based on discrete-time state observation with a time delay},
journal = {Statistics \& Probability Letters},
volume = {115},
pages = {16-26},
year = {2016},
issn = {0167-7152},
author = {Qinwei Qiu and Wei Liu and Liangjian Hu and Xuerong Mao and Surong You}
}

@article{MaoLamHuang2008,
title = {Stabilisation of hybrid stochastic differential equations by delay feedback control},
journal = {Systems \& Control Letters},
volume = {57},
number = {11},
pages = {927-935},
year = {2008},
issn = {0167-6911},
author = {Xuerong Mao and James Lam and Lirong Huang}
}

@article{HuLiuDengMao2020,
author = {Hu, Junhao and Liu, Wei and Deng, Feiqi and Mao, Xuerong},
title = {Advances in Stabilization of Hybrid Stochastic Differential Equations by Delay Feedback Control},
journal = {SIAM Journal on Control and Optimization},
volume = {58},
number = {2},
pages = {735-754},
year = {2020}
}

@article{LiMaoMukamaYuan2020,
author = {Li, Xiaoyue and Mao, Xuerong and Mukama, Denis S. and Yuan, Chenggui},
title = {Delay Feedback Control for Switching Diffusion Systems Based on Discrete-Time Observations},
journal = {SIAM Journal on Control and Optimization},
volume = {58},
number = {5},
pages = {2900-2926},
year = {2020},
}

@article{LiMao2020HighlyNonlinear,
title = {Stabilisation of highly nonlinear hybrid stochastic differential delay equations by delay feedback control},
journal = {Automatica},
volume = {112},
pages = {108657},
year = {2020},
issn = {0005-1098},
author = {Xiaoyue Li and Xuerong Mao}
}

@article{MeiFeiFeiMao2020,
author = {Mei, Chunhui and Fei, Chen and Fei, Weiyin and Mao, Xuerong},
title = {Stabilisation of highly non-linear continuous-time hybrid stochastic differential delay equations by discrete-time feedback control},
journal = {IET Control Theory \& Applications},
volume = {14},
number = {2},
pages = {313-323},
year = {2020}
}

@article{DongTangMao2022,
author = {Dong, Hailing and Tang, Juan and Mao, Xuerong},
title = {Stabilization of Highly Nonlinear Hybrid Stochastic Differential Delay Equations with Lévy Noise by Delay Feedback Control},
journal = {SIAM Journal on Control and Optimization},
volume = {60},
number = {6},
pages = {3302-3325},
year = {2022}
}

@ARTICLE{YuLin2024,
  author={Yu, Xin and Lin, Wei},
  journal={IEEE Transactions on Automatic Control}, 
  title={Sampled-Data Feedback Stabilization in Mean Square for Stochastic Homogeneous Systems}, 
  year={2024},
  volume={69},
  number={10},
  pages={6805-6820}}

@ARTICLE{YuLin2025,
  author={Yu, Xin and Lin, Wei},
  journal={IEEE Transactions on Automatic Control}, 
  title={Almost Sure Asymptotic Stabilization of Stochastic Nonlinear Systems by Sampled-Data State and Output Feedback}, 
  year={2025},
  volume={70},
  number={2},
  pages={1131-1146}}

@ARTICLE{8789495,
  author={Fei, Chen and Fei, Weiyin and Mao, Xuerong and Xia, Dengfeng and Yan, Litan},
  journal={IEEE Transactions on Automatic Control}, 
  title={Stabilization of Highly Nonlinear Hybrid Systems by Feedback Control Based on Discrete-Time State Observations}, 
  year={2020},
  volume={65},
  number={7},
  pages={2899-2912},
 }

@article{mao22_3,
title = {Advances in discrete-state-feedback stabilization of highly nonlinear hybrid systems by Razumikhin technique},
author = {Henglei Xu and Xuerong Mao},
year = {2023},
volume = {68},
pages = {6098--6113},
journal ={IEEE Transactions on Automatic Control},
number ={10},
}

@ARTICLE{LiuTeelSun2022,
  author={Liu, Kun-Zhi and Teel, Andrew R. and Sun, Xi-Ming},
  journal={IEEE Transactions on Automatic Control}, 
  title={Event-Triggered Nonlinear Systems With Stochastic Dynamics, Transmission Times, and Protocols}, 
  year={2022},
  volume={67},
  number={4},
  pages={1973-1979}}

@ARTICLE{LiuTeelSun2024,
  author={Liu, Kun-Zhi and Teel, Andrew R. and Sun, Xi-Ming},
  journal={IEEE Transactions on Automatic Control}, 
  title={A Stochastic Hybrid Framework for Event-Triggered Nonlinear Control Systems With Stochastic Denial of Service Attacks}, 
  year={2024},
  volume={69},
  number={2},
  pages={904-919}}

@article{BakerBuckwar2005,
title = {Exponential stability in p-th mean of solutions, and of convergent Euler-type solutions, of stochastic delay differential equations},
journal = {Journal of Computational and Applied Mathematics},
volume = {184},
number = {2},
pages = {404-427},
year = {2005},
issn = {0377-0427},
author = {Christopher T.H. Baker and Evelyn Buckwar}
}

@article{XuMaoYin2026,
author = {Xu, Henglei and Mao, Xuerong and Yin, George},
title = {Stability and Stabilization Using Discrete-Time Feedback Control for Hybrid Stochastic Delay Systems with General Delay},
journal = {SIAM Journal on Control and Optimization},
volume = {64},
number = {4},
pages = {2844-2869},
year = {2026},
}

@ARTICLE{8897000,
  author={Li, Fengzhong and Liu, Yungang},
  journal={IEEE Transactions on Automatic Control}, 
  title={Event-Triggered Stabilization for Continuous-Time Stochastic Systems}, 
  year={2020},
  volume={65},
  number={10},
  pages={4031-4046},
 }

@ARTICLE{9216607,
  author={Li, Fengzhong and Liu, Yungang},
  journal={IEEE Transactions on Automatic Control}, 
  title={An Enlarged Framework of Event-Triggered Control for Stochastic Systems}, 
  year={2021},
  volume={66},
  number={9},
  pages={4132-4147},
}

@article{liu2025stabilization,
  title={Stabilization of hybrid systems by event-triggered control based on discrete-time state observations},
  author={Liu, Wei and Mao, Wei and Mao, Xuerong and Zhou, Jingchao},
  journal={SIAM Journal on Control and Optimization},
  volume={63},
  number={5},
  pages={3501--3525},
  year={2025},
  publisher={SIAM}
}

@article{zhu2024event,
  title={Event-triggered sampling problem for exponential stability of stochastic nonlinear delay systems driven by L{\'e}vy processes},
  author={Zhu, Quanxin},
  journal={IEEE Transactions on Automatic Control},
  volume={70},
  number={2},
  pages={1176--1183},
  year={2024},
  publisher={IEEE}
}

@ARTICLE{10480561,
  author={Yu, Xin and Lin, Wei},
  journal={IEEE Transactions on Automatic Control}, 
  title={On Input Delay Tolerance of Stochastic Nonlinear Systems With Dominant Homogeneity of Degree Zero}, 
  year={2024},
  volume={69},
  number={10},
  pages={6935-6950},
  }

@article{mao1999stochastic,
  title={Stochastic versions of the LaSalle theorem},
  author={Mao, Xuerong},
  journal={Journal of differential equations},
  volume={153},
  number={1},
  pages={175--195},
  year={1999},
  publisher={Elsevier}
}

@article{MAO1999350,
title = {LaSalle-Type Theorems for Stochastic Differential Delay Equations},
journal = {Journal of Mathematical Analysis and Applications},
volume = {236},
number = {2},
pages = {350-369},
year = {1999},
author = {Xuerong Mao},
}

@article{shim2003asymptotic,
  title={Asymptotic controllability and observability imply semiglobal practical asymptotic stabilizability by sampled-data output feedback},
  author={Shim, Hyungbo and Teel, Andrew R},
  journal={Automatica},
  volume={39},
  number={3},
  pages={441--454},
  year={2003},
  publisher={Elsevier}
}

@ARTICLE{10541043,
  author={Yang, Xuetao and Zhu, Quanxin and Wang, Hua},
  journal={IEEE Transactions on Automatic Control}, 
  title={Exponential Stabilization of Stochastic Systems via Novel Event-Triggered Switching Controls}, 
  year={2024},
  volume={69},
  number={11},
  pages={7948-7955},
  }

@article{cox1996constant,
  title={A note on options and bubbles under the CEV model: implications for pricing and hedging: JC Dias et al.},
  author={Dias, Jos{\'e} Carlos and Nunes, Jo{\~a}o Pedro Vidal and Cruz, Aricson},
  journal={Review of Derivatives Research},
  volume={23},
  number={3},
  pages={249--272},
  year={2020},
  publisher={Springer}
}

@article{ide2002oscillatory,
  title={Oscillatory finite-time singularities in finance, population and rupture},
  author={Ide, Kayo and Sornette, Didier},
  journal={Physica A: Statistical Mechanics and its Applications},
  volume={307},
  number={1-2},
  pages={63--106},
  year={2002},
  publisher={Elsevier}
}
\end{sloppypar}
\end{document}